\documentclass[11pt]{amsart}

\usepackage{amsmath,amsthm,amssymb}
\usepackage[margin=1in]{geometry}
\usepackage[bookmarks]{hyperref}
\usepackage[all,cmtip]{xy}
\usepackage{verbatim}
\usepackage{todonotes}

\newenvironment{enumalph}{\begin{enumerate}  }{\end{enumerate}}

\newcommand{\subgrp}[1]{\langle #1 \rangle}
\newcommand{\set}[1]{\left\{ #1 \right\}}

\newtheorem{theorem}{Theorem}[subsection]
\newtheorem{proposition}[theorem]{Proposition}

\newtheorem{corollary}[theorem]{Corollary}
\newtheorem{lemma}[theorem]{Lemma}

\theoremstyle{definition}

\newtheorem{example}[theorem]{Example}

\newtheorem{remark}[theorem]{Remark}

\newcommand{\upA}{\mathsf{A}}

\newcommand{\C}{\mathbb{C}}
\newcommand{\E}{\mathbb{E}}
\newcommand{\F}{\mathbb{F}}
\newcommand{\N}{\mathbb{N}}

\newcommand{\R}{\mathbb{R}}
\newcommand{\U}{\mathbb{U}}
\newcommand{\Z}{\mathbb{Z}}

\newcommand{\calA}{\mathcal{A}}
\newcommand{\calB}{\mathcal{B}}
\newcommand{\calU}{\mathcal{U}}
\newcommand{\calZ}{\mathcal{Z}}

\newcommand{\frakb}{\mathfrak{b}}
\newcommand{\g}{\mathfrak{g}}

\newcommand{\frakl}{\mathfrak{l}}

\newcommand{\frakp}{\mathfrak{p}}
\newcommand{\fraku}{\mathfrak{u}}

\DeclareMathOperator{\ch}{ch}

\DeclareMathOperator{\Dist}{Dist}

\DeclareMathOperator{\Ext}{Ext}

\DeclareMathOperator{\gr}{gr}
\DeclareMathOperator{\opH}{H}
\newcommand{\Hbul}{\opH^\bullet}
\DeclareMathOperator{\Hom}{Hom}
\DeclareMathOperator{\im}{im}
\DeclareMathOperator{\ind}{ind}
\DeclareMathOperator{\Lie}{Lie}

\newcommand{\CZbar}{\overline{C}_\Z}
\newcommand{\Fp}{\F_p}

\newcommand{\Ua}{U_\upA}

\newcommand{\Uc}{U_\C}
\newcommand{\Uq}{\U_q}
\newcommand{\Uqg}{\U_q(\g)}

\newcommand{\Uz}{U_\zeta}
\newcommand{\Uzb}{\Uz(\frakb)}
\newcommand{\Uzg}{\Uz(\g)}
\newcommand{\Uzo}{\Uz^0}
\newcommand{\Uzlj}{\Uz(\frakl_J)}
\newcommand{\Uzpj}{\Uz(\frakp_J)}

\newcommand{\cua}{\calU_\upA}
\newcommand{\cuc}{\calU_\C}
\newcommand{\cuz}{\calU_\zeta}
\newcommand{\cuzg}{\cuz(\g)}
\newcommand{\cuzu}{\cuz(\fraku)}
\newcommand{\cuzb}{\cuz(\frakb)}
\newcommand{\cuzlj}{\cuz(\frakl_J)}
\newcommand{\cuzuj}{\cuz(\fraku_J)}

\newcommand{\uz}{u_\zeta}
\newcommand{\uzb}{\uz(\frakb)}
\newcommand{\uzg}{\uz(\g)}
\newcommand{\uzo}{\uz^0}
\newcommand{\uzlj}{\uz(\frakl_J)}
\newcommand{\uzpj}{\uz(\frakp_J)}
\newcommand{\uzu}{\uz(\fraku)}

\newcommand{\Lz}{L^\zeta}

\numberwithin{equation}{subsection}

\begin{document}

\title[Cohomology for infinitesimal unipotent algebraic and quantum groups]{Cohomology for infinitesimal unipotent \\ algebraic and quantum groups}

\author{Christopher M. Drupieski}
\address{Department of Mathematics\\ University of Georgia \\ Athens\\ GA~30602, USA}
\thanks{Research of the first author was supported in part by NSF VIGRE grant DMS-0738586}
\email{cdrup@math.uga.edu}

\author{Daniel K. Nakano}
\address{Department of Mathematics\\ University of Georgia \\ Athens\\ GA~30602, USA}
\thanks{Research of the second author was supported in part by NSF grant DMS-1002135}
\email{nakano@math.uga.edu}

\author{Nham V. Ngo}
\address{Department of Mathematics\\ University of Georgia \\ Athens\\ GA~30602, USA}
\thanks{Research of the third author was supported in part by NSF grant DMS-1002135}
\email{nham@uga.edu}



\begin{abstract}
In this paper we study the structure of cohomology spaces for the Frobenius kernels of unipotent and parabolic algebraic group schemes and of their quantum analogs. Given a simple algebraic group $G$, a parabolic subgroup $P_J$, and its unipotent radical $U_J$, we determine the ring structure of the cohomology ring $\opH^\bullet((U_J)_1,k)$. We also obtain new results on computing $\opH^\bullet((P_J)_1,L(\lambda))$ as an $L_J$-module where $L(\lambda)$ is a simple $G$-module with highest weight $\lambda$ in the closure of the bottom $p$-alcove. Finally, we provide generalizations of all our results to small quantum groups at a root of unity.
\end{abstract}

\maketitle

\section{Introduction}

\subsection{}

In the investigation of cohomology for general algebraic structures like groups and Lie algebras, it is important to understand the cohomology of unipotent and nilpotent algebraic structures. For example, it is well-known that the cohomology of a finite group in characteristic $p > 0$ embeds in the cohomology of any one of its $p$-Sylow subgroups. On the other hand, the cohomology of general $p$-groups is not well-understood, which makes the task of computing cohomology for finite groups challenging. In particular, it is not even known how to compute the cohomology of the subgroup of the finite general linear group $GL_n(\Fp)$ consisting of upper triangular unipotent matrices.

For the case of a Lie algebra over a field $k$ of characteristic $p > 0$, far more is known. For example, let $\g$ be the Lie algebra of a simple algebraic group $G$ over $k$, let $\frakp_J \subset \g$ be a standard parabolic subalgebra of $\g$ corresponding to a subset of simple roots $J$, and let $\fraku_J \subset \frakp_J$ be the nilradical of $\frakp_J$. If $p \geq h-1$ (where $h$ is the Coxeter number of $\g$), then an analog of Kostant's famous cohomology formula applies, and one can compute the ordinary Lie algebra cohomology $\opH^\bullet(\fraku_J,L(\lambda))$ with coefficients in a simple $G$-module $L(\lambda)$ having highest weight in the bottom $p$-alcove (cf.\ \cite{FP:1986,PT:2002,UGA:2009}). Moreover, in the case when $L(\lambda)$ is the trivial module $k$ and $J = \emptyset$, then the ring structure of $\opH^\bullet(\fraku,k)$ is also known.

In this paper we demonstrate how to compute cohomology for the infinitesimal Frobenius kernels of unipotent and parabolic algebraic group schemes and of their quantum analogs (i.e., small quantum groups). Our calculations include several new ideas and are highly dependent on the aforementioned calculations for ordinary Lie algebra cohomology and on analogous calculations for quantized enveloping algebras at an $\ell$-th root of unity (cf.\ \cite{UGA:2010}). To compute cohomology for the Frobenius kernels of parabolic algebraic group schemes, we also rely on deep geometric results concerning the vanishing of line bundle cohomology for the flag variety.

\subsection{Main results}

The paper is organized as follows; for a detailed explanation of the notation, see Section~\ref{subsection:notation}. Throughout, assume $p>h$, and let $U_J$ be the unipotent radical of the parabolic subgroup $P_J = L_J \ltimes U_J$ ($J$ a subset of simple roots). Write $(U_J)_1$ for the first Frobenius kernel of $U_J$. In Section \ref{section:UJ1cohomology}, we prove for $\lambda \in C_\Z$ that there exists an isomorphism of graded $L_J$-modules
\[
\opH^\bullet((U_J)_1,L(\lambda))\cong S^\bullet(\fraku_J^*)^{(1)} \otimes\opH^\bullet(\fraku_J,L(\lambda)).
\]
Here $S^\bullet(\fraku_J^*)$ is the symmetric algebra on $\fraku_J^* :=\Hom_k(\fraku_J,k)$, generated in degree two. Our result strengthens an observation previously made by Friedlander and Parshall \cite[Remark 2.7(b)]{FP:1986}. In Section \ref{section:UJ1ringstructure}, we prove for $p>2(h-1)$ that there exists a graded algebra isomorphism 
\[
\opH^\bullet(U_1,k) \cong S^\bullet(\fraku^*)^{(1)}\otimes \opH^\bullet(\fraku,k),
\]
where the algebra structure on the right-hand side is the ordinary tensor product of algebras. This computation answers a 25-year-old problem concerning the ring structure of $\opH^\bullet(U_1,k)$. As $T$-modules, this identification was obtained by Friedlander and Parshall in 1986 provided $p > h$ \cite[Corollary 2.6]{FP:1986}. To obtain the identification of graded algebras, however, the lower bound on $p > 2(h-1)$ cannot in general be improved, as seen from Example \ref{example:counterex}.\footnote{In her thesis, Crane \cite{Crane:1983} claims for all primes that if the underlying root system is of type $A_n$, then the cohomology ring $\opH^\bullet(U_1,k)$ is an integral extension of a polynomial algebra. For $p < h$ this is easily seen to be false using well-established results on the spectrum of the cohomology ring and support varieties. In this paper we verify Crane's claim for $p > 2(h-1)$.} In Section \ref{section:UJ1ringstructure} we also obtain for $p > 3(h-1)$ the ring structure of the cohomology ring $\opH^\bullet((U_J)_1,k)$ for arbitrary $J \subseteq \Delta$.

In Section \ref{section:paraboliccohomology}, we apply our calculations to obtain new results on the cohomology of $(P_J)_1$ with coefficients in a simple $G$-module having highest weight in the bottom $p$-alcove. One of the primary ingredients is the calculation for $p > h$ by Kumar, Lauritzen and Thomsen \cite{KLT:1999} of the cohomology of $G_1$ with coefficients in an induced module, which employs the existence of Frobenius splittings on the cotangent bundle of the flag variety. Our calculations for $(P_J)_1$ significantly extend the earlier calculations of Friedlander and Parshall \cite[Corollary 2.6 and Remark 2.7(b)]{FP:1986}.

Section \ref{section:quantumanalogs} contains quantum analogs of our results for infinitesimal Frobenius kernels. These results pertain to computing the cohomology of the small quantum groups $\uz(\frakp_J)$ and $\uz(\fraku_J)$ with parameter specialized to a primitive $\ell$-th root of unity $\zeta \in \C$. Different techniques are required here than for algebraic groups due to the lack of quantum analogs for various tools available in the study of algebraic groups. 

\subsection{Notation} \label{subsection:notation}

Let $k$ be an algebraically closed field of characteristic $p > 2$. Let $G$ be a simple, simply-connected algebraic group over $k$, defined and split over the prime field $\F_p$. Fix a maximal torus $T \subset G$, also split over $\Fp$, and let $\Phi$ be the root system of $T$ in $G$. Fix a set $\Delta = \{ \alpha_1,\ldots,\alpha_n \}$ of simple roots in $\Phi$, and let $\Phi^+$ be the corresponding set of positive roots.

Let $W$ be the Weyl group of $\Phi$; it is generated by the set of simple reflections $\set{s_\alpha: \alpha \in \Delta}$. Write $\ell: W \rightarrow \N$ for the standard length function on $W$, and let $w_0 \in W$ be the longest element. Let $(\cdot,\cdot)$ be the standard $W$-invariant inner product on the Euclidean space $\E := \Z \Phi \otimes_\Z \R$. Given $\alpha \in \Phi$, let $\alpha^\vee := 2\alpha/(\alpha,\alpha)$ be the corresponding coroot. Set $\alpha_0$ to be the highest short root of $\Phi$, and $\rho$ to be one-half the sum of all positive roots in $\Phi$. Then the Coxeter number of $\Phi$ is $h=(\rho,\alpha^\vee_0)+1$.

Let $X$ be the weight lattice of $\Phi$, defined by the $\Z$-span of the fundamental weights $\{ \omega_1,\ldots,\omega_n \}$, and let $X^+ \subset X$ be the set of dominant weights. The dot action of $W$ on $X$ is defined for $w \in W$ and $\lambda \in X$ by $w\cdot\lambda=w(\lambda+\rho)-\rho$. The bottom $p$-alcove and its closure are defined, respectively, by
\begin{align*}
C_\Z & := \set{ \lambda \in X : 0 < (\lambda+\rho, \beta^\vee) < p \; \text{ for all $\beta \in \Phi^+$}}, \\
\CZbar &:=\set{ \lambda \in X : 0 \leq (\lambda+\rho,\beta^\vee) \leq p \; \text{ for all $\beta \in \Phi^+$}}.
\end{align*}

Given $J \subseteq \Delta$, let $\Phi_J = \Z J \cap \Phi$ be the subroot system of $\Phi$ generated by $J$, and let $W_J \subseteq W$ be the standard parabolic subgroup generated by the set of simple reflections $\set{s_\beta: \beta \in J}$. Set $\Phi_J^+ = \Phi_J \cap \Phi^+$. The set of $J$-dominant weights is defined by $X_J^+ = \set{ \mu \in X: \forall \beta \in \Phi_J^+, (\mu,\beta^\vee) \in \N}$, and the set of $J$-restricted dominant weights is defined by $(X_J)_1 = \set{\mu \in X_J^+: \forall \beta \in J, (\mu,\beta^\vee) < p}$. Write $^JW$ for the set of the minimal length right coset representatives of $W_J$ in $W$. Then ${}^JW$ also satisfies $^JW=\{w\in W|w^{-1}(\Phi^+_J) \subseteq \Phi^+\}$.

Let $B \subset G$ be the Borel subgroup of $G$ containing $T$ and corresponding to $\Phi^+$, and let $U \subset B$ be the unipotent radical of $B$. Write $U^- \subset B^-$ for the opposite subgroups. Given $J \subseteq \Delta$, let $P_J$ be the standard parabolic subgroup of $G$ containing $B$ and corresponding to $J$, let $U_J$ be the unipotent radical of $P_J$, and let $L_J$ be the Levi factor of $P_J$. Then $P_J = L_J\ltimes U_J$. Set $\g = \Lie(G)$, the Lie algebra of $G$, $\frakb = \Lie(B)$, $\fraku = \Lie(U)$, $\frakp_J = \Lie(P_J)$, $\fraku_J = \Lie(U_J)$, and $\frakl_J = \Lie(L_J)$.

Let $F : G \rightarrow G$ be the Frobenius morphism of $G$. Given a closed $F$-stable subgroup (scheme) $H$ of $G$, write $H_1$ for the scheme-theoretic kernel of the morphism $F|_H: H \rightarrow H$ (i.e., the first Frobenius kernel of $H$). Given a rational $H$-module $M$, write $M^{(1)}$ for the module obtained by twisting the structure map for $M$ by $F$. Alternately, if $F$ induces an isomorphism $H/H_1 \cong H$ (a condition that is satisfied for all of the algebraic groups mentioned in the previous paragraph), and if $M$ is an $H/H_1$-module, write $M^{(-1)}$ for the space $M$ considered as an $H$-module via the isomorphism $H/H_1 \cong H$. 

For $\lambda\in X^{+}$, let $L(\lambda)$ be the simple $G$-module of highest weight $\lambda$; it is the $G$-socle of the induced module $H^0(\lambda)=\ind_{B^-}^{G}\lambda$. Similarly, given $J \subseteq \Delta$ and $\lambda \in X_J^+$, write $L_J(\lambda)$ for the simple $L_J$-module of highest weight $\lambda$.

\section{\texorpdfstring{Cohomology of $(U_J)_1$ with coefficients in $L(\lambda)$} {Cohomology of UJ1 with coefficients in L(lambda)}} \label{section:UJ1cohomology}

\subsection{Weight combinatorics and the Weyl group}

We begin with some elementary observations concerning weights. The first lemma is a restatement of \cite[Lemma 13.2A]{Hum:1978}, and the second lemma expresses the well-known fact that the affine Weyl group $W_p$ acts simply transitively on the set of alcoves in $\E$ for $W_p$ \cite[II.6.2(4)]{Jan:2003}.

\begin{lemma} \label{lemma:dotactionfaithful}
Let $w_1,w_2 \in W$, and let $\lambda \in X^+$. If $w_1 \cdot \lambda = w_2 \cdot \lambda$, then $w_1 = w_2$.
\end{lemma}

\begin{lemma} \label{lemma:dotactionmodp}
Let $w_1,w_2 \in W$, $\lambda \in C_\Z$, and $\sigma \in \Z \Phi$. If $w_1 \cdot \lambda = w_2 \cdot \lambda + p \sigma$, then $\sigma = 0$.
\end{lemma}

\begin{remark}
The conclusion Lemma \ref{lemma:dotactionmodp} need not hold if the weight $\lambda$ is merely in the closure $\CZbar$ of the bottom $p$-alcove. For example, suppose $p = 5$ and $\Phi$ is of type $A_2$. Then $\lambda := 2\omega_1 + \omega_2 \in \CZbar$, but $w_0 \cdot \lambda = e \cdot \lambda + 5(-\alpha_1-\alpha_2)$, where $e$ denotes the identity element of $W$.
\end{remark}

\begin{lemma} \label{lemma:sumofdotactionmodp}
Let $w_1,w_2,w_3 \in W$, and suppose $w_1 \cdot 0 + w_2 \cdot 0 = w_3 \cdot 0 + p\sigma$ for some $\sigma \in \Z\Phi$. If $p > 2(h-1)$, then $\sigma = 0$.
\end{lemma}

\begin{proof}
Suppose $w_1 \cdot 0 + w_2 \cdot 0 = w_3 \cdot 0 + p\sigma$ for some $\sigma \in \Z\Phi$. Then $(w_3^{-1}w_1) \cdot 0 + w_3^{-1}(w_2 \cdot 0) = p(w_3^{-1}\sigma)$. Choose $y \in W$ such that $yw_3^{-1} \sigma \in X^+$. Then
\[
y( (w_3^{-1}w_1) \cdot 0) + yw_3^{-1}(w_2 \cdot 0) = p(yw_3^{-1} \sigma) \in \Z \Phi \cap X^+.
\] 
Set $w_1' = w_3^{-1}w_1$, $y' = yw_3^{-1}$, and $\sigma' = yw_3^{-1} \sigma$. Then
\begin{equation} \label{eq:adjointrepweights}
y(w_1' \cdot 0) + y'(w_2 \cdot 0) = p\sigma' \in \Z\Phi \cap X^+.
\end{equation}
Now, $w_1' \cdot 0$ and $w_2 \cdot 0$ are each sums of distinct roots in $\Phi$. Then the same is true for $y(w_1' \cdot 0)$ and $y'(w_2 \cdot 0)$, so $p\sigma' \leq 2\rho + 2\rho = 4 \rho$. Taking the inner product with $\alpha_0^\vee$ preserves the inequality, so we get $p(\sigma',\alpha_0^\vee) \leq 4(h-1)$. Now suppose $\sigma \neq 0$. Then also $\sigma' \neq 0$, hence $\sigma'$ is a nonzero dominant weight in the root lattice. If $(\sigma',\alpha^\vee) \leq 1$ for all $\alpha \in \Phi^+$, then this would imply that $\sigma'$ is a minuscule dominant weight \cite[Exercise 13.13]{Hum:1978}, a contradiction because the minuscule dominant weights are not in the root lattice. Then $(\sigma',\alpha^\vee) \geq 2$ for some $\alpha \in \Phi^+$. Since $\alpha_0^\vee$ is the unique highest root in the dual root system $\Phi^\vee$, this implies that also $(\sigma',\alpha_0^\vee) \geq 2$. Now $2p \leq p(\sigma', \alpha_0^\vee) \leq 4 (h-1)$, hence $p \leq 2(h-1)$. So if $p > 2(h-1)$, then necessarily $\sigma = 0$.
\end{proof}

\begin{example} \label{example:B2weights}
The conclusion of Lemma \ref{lemma:sumofdotactionmodp} need not hold if $p < 2(h-1)$. Indeed, suppose that $p=5$, and that $\Phi$ is of type $B_2$, so that $h=4$. Write $\Delta = \set{\alpha,\beta}$ with $\alpha$ a long root. Then $(s_\beta s_\alpha) \cdot 0 = - \alpha - 3 \beta$ and $(s_\alpha s_\beta) \cdot 0 = -2\alpha - \beta$, so that $(s_\beta s_\alpha) \cdot 0 + (s_\beta s_\alpha) \cdot 0 = (s_\alpha s_\beta) \cdot 0 + 5(-\beta)$.
\end{example}

The following lemma generalizes Lemma \ref{lemma:sumofdotactionmodp}.

\begin{lemma} \label{lemma:adjointrepweightsmodp}
Let $J \subseteq \Delta$. For $i \in \set{1,2,3}$, let $w_i \in {}^JW$, and let $\mu_i$ be a weight of $T$ in $L_J(w_i \cdot 0)$. Suppose $p > 3(h-1)$, and that $\mu_1 + \mu_2 = \mu_3 + p \sigma$ for some $\sigma \in \Z\Phi$. Then $\sigma = 0$.
\end{lemma}

\begin{proof}
Suppose $\mu_1 + \mu_2 = \mu_3 + p\sigma$ for some $\sigma \in \Z\Phi$. Choose $w \in W$ such that $\sigma':=w\sigma \in X^+$, and set $\mu_i' = w\mu_i$. Then $p\sigma' = \mu_1'+\mu_2'-\mu_3'$. The module $L_J(w_i \cdot 0)$ occurs as an $L_J$-composition factor of the cohomology ring $\opH^\bullet(\fraku_J,k)$ \cite[Theorem 4.2.1]{UGA:2009}, which can be computed as a subquotient of the exterior algebra $\Lambda^\bullet(\fraku_J^*)$. The weights of $T$ in $\Lambda^\bullet(\fraku_J^*)$ are also weights of $T$ in the rational $G$-module $\Lambda^\bullet(\g^*)$, which has highest weight $2\rho$ and lowest weight $-2\rho$. Then $\mu_i'$ is also weight of $T$ in $\Lambda^\bullet(\g^*)$, and $-2\rho \leq \mu_i' \leq 2\rho$. This implies that $p\sigma' \leq 6\rho$, and hence that $p(\sigma',\alpha_0^\vee) \leq 6(\rho,\alpha_0^\vee)$. Now if $\sigma \neq 0$, we get as in the proof of Lemma \ref{lemma:sumofdotactionmodp} that $2p \leq 6(h-1)$, that is, $p \leq 3(h-1)$, so if $p > 3(h-1)$, then necessarily $\sigma = 0$.
\end{proof}

\subsection{\texorpdfstring{$L_J$}{LJ}-module structure}

We now compute $\opH^\bullet((U_J)_1,L(\lambda))$ as a graded $L_J$-module.

\begin{theorem} \label{theorem:UJ1computation}
Let $\lambda \in C_\Z\cap X^{+}$. Fix $J \subseteq \Delta$. There exists an isomorphism of graded $L_J$-modules
\begin{equation} \label{eq:UJ1computation}
\opH^\bullet((U_J)_1,L(\lambda)) \cong S^\bullet(\fraku_J^*)^{(1)} \otimes \opH^\bullet(\fraku_J,L(\lambda)),
\end{equation}
where $\fraku_J^*$ in $S^\bullet(\fraku_J^*)$ has cohomological degree 2.
\end{theorem} 

\begin{proof}
By \cite[Proposition 1.1]{FP:1986}, there exists a spectral sequence of $L_J$-modules
\begin{equation} \label{eq:FPspectralsequence}
E_2^{2i,j} = S^i(\fraku_J^*)^{(1)} \otimes \opH^j(\fraku_J,L(\lambda)) \Rightarrow \opH^{2i+j}((U_J)_1,L(\lambda))
\end{equation}
for which $E_2^{i,j} = 0$ for all odd $i$. We claim that the spectral sequence collapses at the $E_2$-page, that is, that $E_2^{i,j} \cong E_\infty^{i,j}$. Using the derivation property of the differential, it suffices to show for all $r \geq 2$ and $j \geq 0$ that $d_r^{0,j} = 0$. This we will do by showing that any $L_J$-composition factor of $E_r^{0,j}$ cannot be a composition factor of $E_r^{r,j+1-r}$. Since $E_r^{i,j}$ is a subquotient of $E_2^{i,j}$, it is enough to show that any composition factor of $E_2^{0,j}$ cannot be a composition factor of $E_2^{r,j+1-r}$.

By \cite[Theorem 2.5]{FP:1986} or \cite[Theorem 4.2.1]{UGA:2009}, there exists an $L_J$-module isomorphism
\begin{equation} \label{eq:Kostantstheorem}
\opH^j(\fraku_J,L(\lambda)) \cong \bigoplus_{\substack{ w \in {}^JW \\ \ell(w) = j}} L_J(w \cdot \lambda).
\end{equation}
Thus, every composition factor of $E_2^{0,j}$ has the form $L_J(w_1 \cdot \lambda)$ for some $w_1 \in {}^JW$ with $\ell(w_1) = j$. Similarly, every composition factor of $E_2^{r,j+1-r} = S^{r/2}(\fraku_J^*)^{(1)} \otimes \opH^{j+1-r}(\fraku_J,L(\lambda))$ must have the form $L_J(\sigma)^{(1)} \otimes L_J(w_2 \cdot \lambda)$ for some $w_2 \in {}^JW$ with $\ell(w_2) = j+1-r$ (in particular, $\ell(w_2) < \ell(w_1)$), and some $\sigma \in X_J^+ \cap \N(\Phi^- \backslash \Phi_J^-)$ (i.e., $\sigma$ is a sum of negative roots not in $\Phi_J^-$).

Observe that for $\alpha \in J$, $(w_2\cdot\lambda,\alpha^\vee) = (\lambda+\rho,w_2^{-1}(\alpha^\vee))-1 \geq 0$, because $w_2 \in {}^JW$ implies that $w_2^{-1}(\alpha) \in \Phi^+$. Also, $(w_2 \cdot \lambda,\alpha^\vee) = (\lambda+\rho,w_2^{-1}(\alpha^\vee))-1 < p$, because $\lambda \in C_\Z$. Then $w_2 \cdot \lambda$ is a $J$-restricted dominant weight, so by the Steinberg Tensor Product Theorem, there exists an $L_J$-module isomorphism
\begin{equation} \label{eq:SteinbergTensorProduct}
L_J(\sigma)^{(1)} \otimes L_J(w_2 \cdot \lambda) \cong L_J(w_2 \cdot \lambda + p\sigma).
\end{equation}
Now suppose that $L_J(w_1 \cdot \lambda) \cong L_J(w_2 \cdot \lambda + p\sigma)$. Then $w_1 \cdot \lambda = w_2 \cdot \lambda + p\sigma$, so Lemmas \ref{lemma:dotactionmodp} and \ref{lemma:dotactionfaithful} imply that $\sigma = 0$ and $w_1 = w_2$. This contradicts the inequality $\ell(w_2) < \ell(w_1)$, so we conclude that no composition factor of $E_2^{0,j}$ can also be a composition factor of $E_2^{r,j+1-r}$, and hence that the spectral sequence \eqref{eq:FPspectralsequence} collapses at the $E_2$-page.

We have shown for all $i$ and $j$ that $E_2^{i,j} \cong E_\infty^{i,j}$. Recall that the $E_\infty$-page of \eqref{eq:FPspectralsequence} is the associated graded module coming from some $L_J$-submodule filtration of $\Hbul((U_J)_1,L(\lambda))$. To finish the proof of the theorem, we must show that $\Hbul((U_J)_1,L(\lambda))$ is isomorphic as an $L_J$-module to its associated graded object. For this, it suffices to show for all $m \neq n$ that
\begin{equation} \label{eq:Ext1LJvanish}
\Ext_{L_J}^1(S^\bullet(\fraku_J^*)^{(1)} \otimes \opH^n(\fraku_J,L(\lambda)), S^\bullet(\fraku_J^*)^{(1)} \otimes \opH^m(\fraku_J,L(\lambda))) = 0.
\end{equation}
Using the long exact sequence in cohomology, and applying the isomorphisms \eqref{eq:Kostantstheorem} and \eqref{eq:SteinbergTensorProduct}, it suffices even to show that $\Ext_{L_J}^1(L_J(w_1 \cdot \lambda+p\sigma),L_J(w_2 \cdot \lambda + p\mu)) = 0$ whenever $w_1,w_2 \in {}^JW$, $w_1 \neq w_2$, and $\mu,\sigma \in \N(\Phi^- \backslash \Phi_J^-)$. So suppose $\Ext_{L_J}^1(L_J(w_1 \cdot \lambda+p\sigma),L_J(w_2 \cdot \lambda + p\mu)) \neq 0$. Then by the Linkage Principle for $L_J$ \cite[II.6.17]{Jan:2003}, there exists $w \in W_J$ and $\gamma \in \Z\Phi_J$ such that
\[
w_1 \cdot \lambda +p\sigma = w \cdot (w_2 \cdot \lambda + p\mu)+p\gamma = (ww_2) \cdot \lambda + p(w\mu+\gamma).
\]
By Lemmas \ref{lemma:dotactionmodp} and \ref{lemma:dotactionfaithful}, this implies that $w_1 \cdot \lambda = (ww_2) \cdot \lambda$, and hence that $w_1 = ww_2$. This is a contradiction, because $w_1$ and $w_2$ were chosen as distinct minimal length right coset representatives of $W_J$ in $W$. So we conclude that $\Ext_{L_J}^1(L_J(w_1 \cdot \lambda+p\sigma),L_J(w_2 \cdot \lambda + p\mu)) = 0$, as claimed.
\end{proof}

\begin{remark}
Friedlander and Parshall observed \eqref{eq:UJ1computation} for $\lambda \in X^+$ satisfying $p > (\lambda,\alpha^\vee_0)+h$ \cite[Remark 2.7(b)]{FP:1986}. Our result strengthens this to $p \geq (\lambda,\alpha_0^\vee)+h$, i.e., for $\lambda \in C_\Z \cap X^+$. 
\end{remark}

\subsection{Structure of the associated graded algebra}

Given an algebra $A$ with a multiplicative filtration indexed by the nonnegative integers, let $\gr A$ be the associated graded algebra.

\begin{corollary} \label{corollary:UJ1ringcomputation}
Fix $J \subseteq \Delta$, and suppose $p > h$. Then there exists a multiplicative filtration on $\opH^\bullet((U_J)_1,k)$ by $L_J$-submodules such that
\begin{equation} \label{eq:gradedalgebraiso}
\gr \opH^\bullet((U_J)_1,k) \cong S^\bullet(\fraku_J^*)^{(1)} \otimes \opH^\bullet(\fraku_J,k)
\end{equation}
as graded $L_J$-algebras, where $\fraku_J^*$ in $S^\bullet(\fraku_J^*)$ has cohomological degree 2, and the algebra structure on the right-hand side of the isomorphism is the standard tensor product of algebras.
\end{corollary}

\begin{proof}
Take $\lambda = 0$ in Theorem \ref{theorem:UJ1computation}. Then $L(\lambda) = k$, \eqref{eq:FPspectralsequence} is a spectral sequence of algebras, and \eqref{eq:gradedalgebraiso} is the algebra isomorphism between the $E_2$- and $E_\infty$-pages of \eqref{eq:FPspectralsequence}.
\end{proof}

\begin{remark} \label{remark:polynomialfiltration}
The $L_J$-submodule filtration on $\Hbul((U_J)_1,k)$ described in Corollary \ref{corollary:UJ1ringcomputation} can be given explicitly as follows: For $i \geq 0$, let $F^i \Hbul((U_J)_1,k)$ denote the $i$-th filtered part of $\Hbul((U_J)_1,k)$. Then there exists an $L_J$-module isomorphism
\[
F^i \opH^n((U_J)_1,k) \cong \bigoplus_{j \geq i} S^{j/2}(\fraku_J^*)^{(1)} \otimes \opH^{n-j}(\fraku,k),
\]
where both sides are zero if $i > n$.
\end{remark}

\begin{remark} \label{remark:ordinaryringstructure}
Suppose $p \geq h-1$. Implicit in \cite{UGA:2009} is the following description of the ring structure on $\opH^\bullet(\fraku,k)$. First, $\opH^\bullet(\fraku,k)$ is computed as a subquotient of the Koszul complex $\Lambda^\bullet(\fraku^*)$ (the exterior algebra on $\fraku^*$). Fix an ordered root vector basis $\set{x_{\gamma_1},\ldots,x_{\gamma_N}}$ for $\fraku$, and let $\set{f_{\gamma_1},\ldots,f_{\gamma_N}} \subset \fraku^*$ be the corresponding dual basis. For $w \in W$, define $\Phi(w) = w\Phi^- \cap \Phi^+$. Write $\Phi(w) = \set{\beta_1,\ldots,\beta_n}$ with $n = \ell(w)$, and set $f_{\Phi(w)} = f_{\beta_1} \wedge \cdots \wedge f_{\beta_n} \in \Lambda^n(\fraku^*)$. For definiteness, assume that the list $\set{f_{\beta_1},\ldots,f_{\beta_n}}$ appears as a subsequence of the list $\set{f_{\gamma_1},\ldots,f_{\gamma_N}}$. Then the vectors $f_{\Phi(w)}$ for $w \in W$ are cocycles in $\Lambda^\bullet(\fraku^*)$, and their images $[f_{\Phi(w)}]$ in $\opH^\bullet(\fraku,k)$ form a vector space basis for $\opH^\bullet(\fraku,k)$. The ring structure on $\opH^\bullet(\fraku,k)$ is induced by the ring structure of $\Lambda^\bullet(\fraku^*)$. Specifically, the cup product of cohomology classes is given by
\[
[f_{\Phi(w)}] \cup [f_{\Phi(w')}] = \begin{cases} (-1)^{s(w,w')} [f_{\Phi(w'')}] & \text{if $\ell(w)+\ell(w') = \ell(w'')$ and $\Phi(w) \cup \Phi(w') = \Phi(w'')$}, \\
0 & \text{otherwise.} \end{cases}
\]
If $\Phi(w) = \set{\beta_1,\ldots,\beta_n}$ and $\Phi(w') = \set{\beta_1',\ldots,\beta_m'}$ with $\set{\beta_1,\ldots,\beta_n}$ and $\set{\beta_1',\ldots,\beta_m'}$ written as subsequences of the list $\set{\gamma_1,\ldots,\gamma_N}$, then the integer $(-1)^{s(w,w')}$ is the sign of the permutation that maps the list $\set{\beta_1,\ldots,\beta_n,\beta_1',\ldots,\beta_m'}$ to a subsequence of the list $\set{\gamma_1,\ldots,\gamma_N}$.

\end{remark}

\section{\texorpdfstring{The ring structure of $\opH^\bullet(U_1,k)$}{The ring structure of H(U1,k)}} \label{section:UJ1ringstructure}

\subsection{Un-grading the associated graded algebra}

For $p > h$, the ring structure of the cohomology rings $\opH^\bullet(G_1,k)$ and $\opH^\bullet(B_1,k)$ for the first Frobenius kernels of $G$ and $B$ were computed by Friedlander and Parshall \cite{FP:1986a} and Andersen and Jantzen \cite{AJ:1984}. In this section we compute, for $J \subseteq \Delta$ and for $p$ not too small, the ring structure of the cohomology ring $\opH^\bullet((U_J)_1,k)$ for the first Frobenius kernel of $U_J$. In particular, we compute the ring structure of $\Hbul(U_1,k)$.

\begin{theorem} \label{theorem:U1cohomologyringstructure}
Fix $J \subseteq \Delta$. If $J = \emptyset$, assume $p>2(h-1)$. If $J \neq \emptyset$, assume $p > 3(h-1)$. Then there exists a graded $L_J$-algebra isomorphism
\begin{equation}
\opH^\bullet((U_J)_1,k) \cong S^\bullet(\fraku_J^*)^{(1)} \otimes \opH^\bullet(\fraku_J,k),
\end{equation}
where the algebra structure on the right-hand side is the ordinary tensor product of algebras.
\end{theorem}

\begin{proof}
To simplify the notation slightly, we give the proof for the case $J = \emptyset$. For the case $J \neq \emptyset$, one simply replaces in the following argument the symbols $T$, $U$, and $\fraku$ by $L_J$, $U_J$, and $\fraku_J$, respectively, and applies Lemma \ref{lemma:adjointrepweightsmodp} instead of Lemma \ref{lemma:sumofdotactionmodp}. Our strategy is to exhibit graded $T$-subalgebras $\calA$ and $\calB$ of $\Hbul(U_1,k)$ isomorphic to $S^\bullet(\fraku^*)^{(1)}$ and $\Hbul(\fraku,k)$, respectively, and then to show that the product map $\theta: \calA \otimes \calB \rightarrow \Hbul(U_1,k)$ is an isomorphism of algebras.

First set $\calA = \bigoplus_{n \geq 0} F^n \opH^n(U_1,k)$. Then $\calA$ is a subalgebra of $\Hbul(U_1,k)$, and it follows from Remark \ref{remark:polynomialfiltration} that $\calA$ is isomorphic as an $T$-algebra to $S^\bullet(\fraku^*)^{(1)}$. Next, by Theorem \ref{theorem:UJ1computation} there exists a $T$-submodule $\calB$ of $\Hbul(U_1,k)$ isomorphic to $\Hbul(\fraku,k)$. Specifically, in the $T$-module direct sum decomposition
\begin{equation} \label{eq:Tmoddirectsumdecomp}
\Hbul(U_1,k) \cong \bigoplus_{n \geq 0} \bigoplus_{2i+j = n} S^i(\fraku^*)^{(1)} \otimes \opH^j(\fraku,k),
\end{equation}
the space $\calB$ is the sum of all terms with $i = 0$. By Remark \ref{remark:polynomialfiltration}, $\calB = \Hbul(U_1,k)/F^1 \Hbul(U_1,k)$.

We claim that $\calB$ is a subalgebra of $\Hbul(U_1,k)$. To see this, it suffices to show that the product of two weight vectors in $\calB$ is again a weight vector in $\calB$. So let $z_1,z_2 \in \calB$ be vectors of weights $\mu_1$ and $\mu_2$, respectively, and suppose that $z_1z_2 \neq 0$. Then the product $z_1z_2$ has weight $\mu_1 + \mu_2$. Suppose the product $z_1z_2$ has a nonzero component in some summand $S^i(\fraku^*)^{(1)} \otimes \opH^j(\fraku,k)$ of \eqref{eq:Tmoddirectsumdecomp} with $i \neq 0$. Then there exists a weight $\sigma$ of $S^i(\fraku^*)$ and a weight $\mu_3$ of $\opH^j(\fraku,k)$ such that $\mu_1 + \mu_2 = \mu_3 + p\sigma$. Since $i \neq 0$, we must have $\sigma \neq 0$. But now the explicit description of the weights of $\Hbul(\fraku,k)$ given by \eqref{eq:Kostantstheorem} together with Lemma \ref{lemma:sumofdotactionmodp} imply that $\sigma = 0$, a contradiction. Thus, we must have $z_1z_2 \in \calB$, and so $\calB$ is a subalgebra of $\Hbul(U_1,k)$. To see that $\calB \cong \Hbul(\fraku,k)$ as an algebra, and not just as a $T$-module, observe that since the spectral sequence \eqref{eq:FPspectralsequence} collapses at the $E_2$-page, the vertical edge map $\varphi: \Hbul(U_1,k) \rightarrow E_2^{0,\bullet}$ induces a $T$-algebra isomorphism $\Hbul(U_1,k)/F^1 \Hbul(U_1,k) \cong E_2^{0,\bullet} \cong \Hbul(\fraku,k)$. By the above argument, the quotient $\Hbul(U_1,k)/F^1 \Hbul(U_1,k)$ identifies not only as a space but as an algebra with $\calB$, so we get the $T$-algebra isomorphism $\calB \cong \Hbul(\fraku,k)$.

We have now exhibited $T$-subalgebras $\calA$ and $\calB$ of $\Hbul(U_1,k)$ isomorphic to $S^\bullet(\fraku^*)^{(1)}$ and $\Hbul(\fraku,k)$, respectively. It remains to show that the product map $\theta: \calA \otimes \calB \rightarrow \Hbul(U_1,k)$ is an algebra isomorphism, where the algebra structure on $\calA \otimes \calB$ is the ordinary tensor product of algebras. First, the cohomology ring $\Hbul(U_1,k)$ identifies with the cohomology ring of the finite-dimensional Hopf algebra $\Dist(U_1)$ (cf.\ \cite[I.7, I.9]{Jan:2003}), so is a graded-commutative ring by \cite[VIII.4]{Mac-Lane:1995}\footnote{Mac Lane does not assume a Hopf algebra to possess an antipode. In particular, the cohomology ring of a bialgebra is always graded-commutative.}. Since the subalgebra $\calA$ is concentrated in even degrees, this implies that $\theta$ is an algebra homomorphism. Next, there exist natural maps
\begin{align*}
\iota_1: \calA &= \bigoplus_{n \geq 0} F^n \opH^n(U_1,k) \hookrightarrow \bigoplus_{i \geq 0} F^i \Hbul(U_1,k)/F^{i+1} \Hbul(U_1,k) \quad \text{and} \\
\iota_2: \calB &= \Hbul(U_1,k)/F^1 \Hbul(U_1,k) \hookrightarrow \bigoplus_{i \geq 0} F^i \Hbul(U_1,k)/F^{i+1} \Hbul(U_1,k).
\end{align*}
By Remark \ref{remark:polynomialfiltration}, the images of $\calA$ and $\calB$ under these maps generate the associated graded algebra $\gr \Hbul(U_1,k)$. It follows then that $\calA$ and $\calB$ also generate the cohomology ring $\Hbul(U_1,k)$. Indeed, given a nonzero homogeneous element $z \in \opH^n(U_1,k)$, choose $i$ such that $z \in F^i \opH^n(U_1,k)$ but $z \notin F^{i+1} \opH^n(U_1,k)$. Since $\iota_1(\calA)$ and $\iota_2(\calB)$ generate $\gr \Hbul(U_1,k)$, there exists $w \in \calA \otimes \calB$ such that $z - \theta(w) \in F^{i+1} \opH^n(U_1,k)$. By induction on $i-n$, we may assume that $F^{i+1} \opH^n(U_1,k)$ is in the image of the map $\theta$. Then $z \in \im(\theta)$, so we conclude that $\theta$ is surjective. Finally, by dimension comparison in each graded degree, $\theta$ must also be injective, hence an algebra isomorphism.
\end{proof}

\begin{remark}
The bound of $3(h-1)$ in Theorem \ref{theorem:U1cohomologyringstructure} for the case $J \neq \emptyset$ is not sharp. For example, suppose $\Phi$ has type $A_n$ with $n > 1$, and that $J = \Delta - \set{\alpha}$ for some simple root $\alpha$. Then $U_J$ is abelian, so $\opH^\bullet((U_J)_1,k) \cong S^\bullet(\fraku_J^*)^{(1)} \otimes \Lambda^\bullet(\fraku_J^*) \cong S^\bullet(\fraku_J^*)^{(1)} \otimes \opH^\bullet(\fraku_J,k)$ as a ring if $p > 2$.
\end{remark}

\subsection{Failure for small p}

If $p < 2(h-1)$, the algebra isomorphism $\opH^\bullet(U_1,k) \cong S^\bullet(\fraku^*)^{(1)} \otimes \opH^\bullet(\fraku,k)$ of Theorem \ref{theorem:U1cohomologyringstructure} need not hold, even though the isomorphism of associated graded algebras $\gr \Hbul(U_1,k) \cong S^\bullet(\fraku^*)^{(1)} \otimes \opH^\bullet(\fraku,k)$ holds whenever $p > h$.

\begin{example} \label{example:counterex}
Let $\Phi$ be of type $B_2$, so $h=4$, and take $p=5$. Write $\Delta = \set{\alpha,\beta}$ with $\alpha$ a long root. Then Example \ref{example:B2weights} shows that the weight argument in the proof of Theorem \ref{theorem:U1cohomologyringstructure} fails, as there exist weights $\mu_1 = \mu_2 = (s_\beta s_\alpha) \cdot 0$ and $\mu_3 = (s_\alpha s_\beta) \cdot 0$ of $\opH^2(\fraku,k)$ and a weight $\sigma = -\beta$ of $S^1(\fraku^*)$ such that $\mu_1 + \mu_2 = \mu_3 + p\sigma$ but $\sigma \neq 0$. In fact, this nontrivial solution to the weight equation $\mu_1 + \mu_2 = \mu_3 + p \sigma$ corresponds to two elements in the subspace $\calB$ of $\Hbul(U_1,k)$ having a product not in $\calB$. Indeed, let $z_1 = z_2$ be a nonzero weight vector in the one-dimensional weight space $\opH^2(U_1,k)_{s_\beta s_\alpha \cdot 0} \subset \calB$. We have been able to verify by computer calculation in MAGMA \cite{BCP:1997} that $z_1z_2 \neq 0$ in $\Hbul(U_1,k)$, even though every vector in the Lie algebra cohomology ring $\Hbul(\fraku,k)$ squares to zero. Thus, for type $B_2$ we cannot have an isomorphism of graded $T$-algebras $\opH^\bullet(U_1,k) \cong S^\bullet(\fraku^*)^{(1)} \otimes \opH^\bullet(\fraku,k)$ when $p=5$.
\end{example}

\section{Parabolic cohomology} \label{section:paraboliccohomology}

\subsection{}

In this section we apply our previous results to compute the structure of the cohomology space $\opH^\bullet((P_J)_1,L(\lambda))$ when $\lambda \in X^+ \cap \CZbar$. (Recall that $X^+ \cap \CZbar$ is non-empty by our standing assumption that $p > h$.) First we prove a combinatorial lemma. 

\begin{lemma} \label{l1}
Let $\lambda \in X^+ \cap \CZbar$. Suppose that $\lambda$ is weakly $p$-linked to $0$, that is, that $\lambda= w\cdot 0+p\sigma$ for some $\sigma\in X$. Then:
\begin{enumalph} 
\item The weight $\sigma$ is minuscule or zero. The weight $\lambda$ uniquely determines $w$ and $\sigma$.
\item There exists a $T$-module isomorphism,
\begin{equation*}
\opH^{j}(\fraku,L(\lambda))^{T_1}\cong
\begin{cases}
w^{-1}\sigma & \text{if } j=\ell(w) \\
0 & \text{otherwise}
\end{cases}
\end{equation*}
\end{enumalph} 
\end{lemma} 

\begin{proof}
First, since $w\rho+p\sigma = \lambda +\rho \in \rho +X^+$, and since $|(w\rho,\alpha^\vee)| = |(\rho,w^{-1}\alpha^\vee)| \leq h-1 < p$ for all $\alpha \in \Phi$, it follows that $\sigma \in X^+$. Next, since $\lambda \in \CZbar$, we get
\begin{align*}
p(\sigma,\alpha_0^\vee) &= (\lambda- w \cdot 0,\alpha_0^\vee) \\
&= (\lambda+\rho,\alpha_0^\vee)-(\rho,w^{-1}\alpha_0^\vee) \\
&\leq p+(\rho,\alpha_0^\vee)\\
&= p+h-1 < 2p.
\end{align*}
It follows that $(\sigma,\alpha_0^\vee) \in \set{0,1}$. If $(\sigma,\alpha_0^\vee) = 0$, then $\lambda=0$ and $w=1$. Otherwise, $\sigma$ is a minuscule weight. Now suppose that $\lambda=w\cdot 0+p\omega_i=w'\cdot 0+p\omega_j$ for some $w,w' \in W$ and some minuscule weights $\omega_i,\omega_j$ (so $\Phi$ is necessarily of type $A$, $D$ or $E$). Then $p(\omega_j-\omega_i) = w\cdot 0-w'\cdot 0 \in \Z\Phi$. Since $p > h$, we conclude that $\omega_i = \omega_j$, for otherwise $p(\omega_j-\omega_i) \notin \Z\Phi$. Now $w \cdot 0 = w' \cdot 0$, so $w = w'$ by Lemma \ref{lemma:dotactionfaithful}. This proves part (a). Next, by \cite[Theorem 4.2.1]{UGA:2009} we have 
\[
\opH^j(\fraku,L(\lambda))^{T_1} \cong \bigoplus_{\substack{w' \in W \\ \ell(w') = j}} (k_{w' \cdot \lambda})^{T_1}.
\]
Here $k_{w' \cdot \lambda}$ denotes the one-dimensional $T$-module of weight $w' \cdot \lambda$. Suppose $w' \cdot \lambda = p \sigma'$ for some $\sigma' \in X$. Then $0 = w' \cdot \lambda - p\sigma' = w'w \cdot 0 + p(w'\sigma - \sigma')$. By (a), this implies that $w'\sigma - \sigma' = 0$ and that $w'w = 1$, that is, that $w' = w^{-1}$. So $\sigma' = w^{-1} \sigma$ and $\ell(w) = \ell(w') = j$.
\end{proof}

\subsection{}
We can now compute the structure of $\opH^\bullet((P_J)_1,L(\lambda))$ when $\lambda\in X^+ \cap \CZbar$. 

\begin{theorem} \label{theorem:alggrpparabolic}
Let $\lambda\in X^+\cap \CZbar$ and let $J \subseteq \Delta$. 
\begin{enumalph} 
\item $\opH^\bullet((P_J)_1,L(\lambda))=0$ if $\lambda$ is not weakly $p$-linked to $0$. 
\item If $\lambda=w\cdot 0+p\sigma$, then there exists a $P_J$-module isomorphism
\[
\opH^j((P_J)_1,L(\lambda))^{(-1)} \cong
\begin{cases}
\ind_B^{P_J} [S^{\frac{j-\ell(w)}{2}}(\fraku^*)\otimes w^{-1}\sigma] & \text{if } j\equiv l(w)\ \text{mod }2, \\
0 & \text{otherwise.}
\end{cases}
\]
\end{enumalph}
\end{theorem}

\begin{proof}
Part (a) is established by the argument in \cite[Remark 2.7(b)]{FP:1986}, using the Linkage Principle for $(L_J)_1$. Next suppose that $\lambda = w \cdot 0 + p \sigma$ for some $w \in W$ and $\sigma \in X$. We have $R^m \ind_B^{P_J} L(\lambda) \cong L(\lambda) \otimes R^m \ind_B^{P_J}(k) =0$ for all $m>0$, so by \cite[II.12.2]{Jan:2003}, there exists a spectral sequence
\begin{equation} \label{eq:parabolicspecseq}
E_2^{i,j} = R^i \ind_B^{P_J} [ \opH^j(B_1,L(\lambda))^{(-1)} ] \Rightarrow 
\opH^{i+j}((P_J)_1,L(\lambda))^{(-1)}.
\end{equation}
Also, $\opH^\bullet(B_1,L(\lambda)) \cong \opH^\bullet(U_1,L(\lambda))^{T_1}$ by \cite[I.6.9(3)]{Jan:2003}, so by Lemma \ref{l1}(b), 
\begin{align*}
R^i \ind_B^{P_J} [ \opH^j(B_1,L(\lambda))^{(-1)} ] &\cong R^i \ind_B^{P_J}[(\opH^j(U_1,L(\lambda))^{T_1})^{(-1)}] \\
&\cong R^i \ind_B^{P_J} [(\oplus_{2a+b=j} \, S^a(\fraku^*)^{(1)} \otimes \opH^b(\fraku,L(\lambda))^{T_1})^{(-1)}] \\
&\cong R^i \ind_B^{P_J} [S^{\frac{j-\ell(w)}{2}}(\fraku^*)\otimes w^{-1}\sigma],
\end{align*}
and the spectral sequence can be rewritten as
\[
E_2^{i,j} = R^i \ind_B^{P_J} [S^{\frac{j-\ell(w)}{2}}(\fraku^*) \otimes w^{-1}\sigma] \Rightarrow \opH^{i+j}((P_J)_1,L(\lambda))^{(-1)}.
\]

Next, we claim that $w^{-1}\sigma$ is antidominant, that is, that $w^{-1} \sigma \in -X^+$. To see this, first observe that for any $\alpha \in \Delta$, $(w^{-1} \cdot \lambda,\alpha^\vee) = p(w^{-1}\sigma,\alpha^\vee) \in p\Z$. Next observe that
\[
(w^{-1}\cdot \lambda,\alpha^\vee) = (\lambda + \rho, w \alpha^\vee) - 1 < p
\]
because $\lambda \in \CZbar$. Then for all $\alpha \in \Delta$, $(w^{-1} \cdot \lambda,\alpha^\vee) \leq 0$, so also $(w^{-1}\sigma,\alpha^\vee) \leq 0$, i.e., $w^{-1}\sigma \in -X^+$.

Now set $B' = B \cap L_J$, $U' = U \cap L_J$, and $\fraku' = \Lie(U')$. Then $\fraku = \fraku' \oplus \fraku_J$ and $S^\bullet(\fraku^*) \cong S^\bullet(\fraku'^*)\otimes S^\bullet(\fraku_J^*)$ are $B'$-module decompositions, and the action of $B'$ on $S^\bullet(\fraku_J^*)$ lifts to the natural action of $L_J$ on $S^\bullet(\fraku_J^*)$. By \cite[Example 4.2(a)]{Cline:1983}, there exists an isomorphism of $L_J$-modules
\[
R^i \ind_B^{P_J} [S^\bullet(\fraku^*) \otimes w^{-1}\sigma] \cong R^i \ind_{B'}^{L_J} [S^\bullet(\fraku^*) \otimes w^{-1}\sigma],
\]
and then by the generalized tensor identity we get
\[
R^i \ind_{B'}^{L_J} [S^\bullet(\fraku^*)\otimes w^{-1}\sigma] \cong S^\bullet(\fraku_J^*)\otimes R^i \ind_{B'}^{L_J} [S^\bullet(\fraku'^*)\otimes w^{-1}\sigma]
\]
Now $R^i \ind_{B'}^{L_J} [S^\bullet(\fraku'^*)\otimes w^{-1}\sigma] = 0$ if $i > 0$ by the calculations of Kumar, Lauritzen, and Thomsen \cite[Theorem 2]{KLT:1999}.\footnote{Kumar, Lauritzen, and Thomsen choose their Borel subgroup to correspond to the negative roots, while we have chosen ours to correspond to the positive roots; this is why we checked, for example, that $w^{-1}\sigma \in -X^+$ rather than $w^{-1}\sigma \in X^+$.} Consequently, the spectral sequence \eqref{eq:parabolicspecseq} collapses at the $E_2$-page, and we obtain the $P_J$-module isomorphism $\opH^j((P_J)_1,L(\lambda))^{(-1)} \cong \ind_B^{P_J} [S^{\frac{j-\ell(w)}{2}}(\fraku^*) \otimes w^{-1}\sigma]$.
\end{proof}

\section{Results for Quantum Groups} \label{section:quantumanalogs}

In this section we adapt our main results to quantum groups. New techniques are necessary for quantum groups because of the lack of a quantum analog for the spectral sequence \eqref{eq:FPspectralsequence}. The arguments given here for quantum groups can be adapted to work for algebraic groups as well, and thus provide different proofs of the earlier theorems. The notation we use in this section is generally the same as that in \cite{Bendel:2011}, though to maintain consistency with Sections \ref{section:UJ1cohomology}--\ref{section:paraboliccohomology} we define our Borel subalgebras to correspond to positive root vectors and not negative root vectors.

\subsection{Notation and preliminaries}

Let $q$ be an indeterminate. The quantized enveloping algebra $\Uq = \Uqg$ is the $\C(q)$-algebra defined by the generators $\{E_\alpha,F_\alpha,K_\alpha,K_\alpha^{-1}:\alpha \in \Delta\}$ and the relations in \cite[4.3]{Jantzen:1996}. The Hopf algebra structure maps for $\Uq$ are described in \cite[4.8]{Jantzen:1996}.

Let $\ell$ be an odd positive integer, with $\ell$ coprime to 3 if the root system $\Phi$ of $\g$ is of type $G_2$. Let $\zeta \in \C$ be a primitive $\ell$-th root of unity. Set $\upA = \Z[q,q^{-1}]$. Then the field $\C$ is naturally an $\upA$-module under the specialization $q \mapsto \zeta$. There exist two $\upA$-forms in $\Uq$, the Lusztig divided power integral form $\Ua$ (defined in \cite[\S 1.3]{Lusztig:1990}), and the De~Concini--Kac integral form $\cua$ (defined in \cite[\S 1.5]{De-Concini:1990}). Set $\Uc = \Ua \otimes_\upA \C$, and define $\Uz = \Uzg$ to be the quotient of $\Uc$ by the two-sided ideal generated by the set $\{K_\alpha^\ell \otimes 1 - 1 \otimes 1: \alpha \in \Delta\} \subset \Uc$, i.e.,
\[
\Uz = \Uzg = \Uc/\subgrp{K_\alpha^\ell \otimes 1 - 1 \otimes 1: \alpha \in \Delta}.
\]
Similarly, set $\cuc = \cua \otimes_\upA \C$, and define $\cuz = \cuzg$ to be the quotient of $\cuc$ by the two sided ideal generated by the set $\{K_\alpha^\ell \otimes 1 - 1 \otimes 1: \alpha \in \Delta\} \subset \cuc$.

By abuse of notation, we denote from now on the generators for $\Uq$ as well as their images in $\Uz$ and $\cuz$ by the same symbols. We follow the usual convention (cf.\ \cite[\S 1.4]{Lusztig:1990}) of writing the superscripts $+$, $-$, and $0$ to denote the positive, negative, and toral subalgebras of $\Uq$, $\Uz$, and $\cuz$. The Borel subalgebras of $\Uz$ and $\cuz$ are then defined by $\Uzb = \Uz^+ \Uzo$ and $\cuzb = \cuz^+ \cuz^0$. We also set $\cuzu = \cuz^+$. For $J \subseteq \Delta$, there also exist subalgebras of $\Uz$ and $\cuz$ corresponding to the Lie algebras $\frakp_J$, $\fraku_J$, and $\frakl_J = \Lie(L_J)$; see \cite[\S 2.5]{Bendel:2011}.

The elements $\set{E_\alpha,F_\alpha,K_\alpha:\alpha \in \Delta} \subset \Uz$ generate a finite-dimensional Hopf-subalgebra of $\Uz$, called the small quantum group and denoted $\uz = \uzg$. Set $\uzu = \uz^+$, and write $\uzb = \uz^+ \uzo$ for the Borel subalgebra of $\uzg$. For each positive root $\gamma \in \Phi^+$, there exist root vectors $E_\gamma \in \Uq^+$ and $F_\gamma \in \Uq^-$, defined in terms of certain braid group operators on $\Uq$. Let $\calZ$ be the subalgebra of $\cuz$ generated by the set $\{E_\gamma^\ell,F_\gamma^\ell: \gamma \in \Phi^+ \} \subset \cuz$. Then $\calZ$ is a central polynomial subalgebra of $\cuz$, and $\cuz$ is free and finitely-generated over $\calZ$ \cite[\S 3.1, 3.3]{De-Concini:1990}. The inclusion of $\upA$-forms $\cua \rightarrow \Ua$ induces algebra isomorphisms $\cuz//\calZ \cong \uz$ and $\cuzb//\calZ^+ \cong \uzb$.\footnote{Let $A$ be an augmented $k$-algebra, and let $B$ be a subalgebra of $A$. Let $B_+$ denote the augmentation ideal of $B$. We say that $B$ is normal in $A$ if $AB_+ = B_+A$. In this case, we write $A//B$ for the quotient $A/(AB_+)$. In particular, we use this notation even if $A$ and $B$ are not Hopf algebras.} The quotient map $\cuzb \rightarrow \uzb$ restricts to an isomorphism $\cuz^0 \cong \uzo$.

\begin{lemma} \label{lemma:ellthpowerscoproducts}
The coproduct $\Delta$ of $\cuz$ satisfies, for all $\gamma \in \Phi^+$,
\begin{align*}
\Delta(E_\gamma^\ell) &= E_\gamma^\ell \otimes 1 + 1 \otimes E_\gamma^\ell \qquad \text{and} \\
\Delta(F_\gamma^\ell) &= F_\gamma^\ell \otimes 1 + 1 \otimes F_\gamma^\ell.
\end{align*}
In particular, the subalgebras $\calZ$ and $\calZ^+$ of $\cuz$ are bialgebras.
\end{lemma}

\begin{proof}
Let $\calZ_\C$ be the subalgebra of $\cuc$ generated by $\{ E_\gamma^\ell,F_\gamma^\ell,K_\alpha^{\pm \ell}: \gamma \in \Phi^+, \alpha \in \Delta \}$. For $\alpha \in \Delta$, we have $\Delta(K_\alpha^\ell) = K_\alpha^\ell \otimes K_\alpha^\ell$. If $\gamma \in \Phi^+$ is a simple root, then also
\begin{align}
\Delta(E_\gamma^\ell) &= E_\gamma^\ell \otimes 1 + K_\gamma^\ell \otimes E_\gamma^\ell \quad \text{and} \label{eq:Eellcoproduct} \\
\Delta(F_\gamma^\ell) &= F_\gamma^\ell \otimes K_\gamma^{-\ell} + 1 \otimes F_\gamma^\ell \label{eq:Fellcoproduct}
\end{align}
in $\cuc$ by \cite[4.9(4)]{Jantzen:1996}. The subalgebra $\calZ_\C$ of $\cuc$ is stable under the braid group automorphisms of $\cuc$ by \cite[Proposition 3.3]{De-Concini:1990}, so the identities \eqref{eq:Eellcoproduct} and \eqref{eq:Fellcoproduct} follow for arbitrary $\gamma \in \Phi^+$ from \cite[Proposition C.5(2)]{AJS:1994}. Thus, $\calZ_\C$ and $\calZ_\C^+$ are sub-bialgebras of $\cuc$. The lemma now follows because $\calZ$ and $\calZ^+$ are the images of $\calZ_\C$ and $\calZ_\C^+$ in $\cuz$.
\end{proof}

Let $\U(\g)$ be the universal enveloping algebra for $\g$, and let $F_\zeta: \Uz \rightarrow \U(\g)$ be the quantum Frobenius morphism defined by Lusztig in \cite[\S 8]{Lusztig:1990}. Then $F_\zeta$ induces the Hopf-algebra isomorphism $\Uz//\uz \cong \U(\g)$. Similarly, let $\U(\frakb)$ be the universal enveloping algebra for $\frakb$. Then $F_\zeta$ restricts to a map $\Uzb \rightarrow \U(\frakb)$, and induces the isomorphism $\Uzb//\uzb \cong \U(\frakb)$. Given a left $\U(\g)$- (resp.\ $\U(\frakb)$)-module $M$, write $M^{(1)}$ for $M$ considered as a $\Uz$- (resp.\ $\Uzb$)-module via $F_\zeta$.

Given $\lambda \in X^+$, let $\Lz(\lambda)$ be the irreducible integrable $\Uz$-module (i.e., the irreducible type-1 integrable $\Uc$-module) of highest weight $\lambda$. Similarly, given $\lambda \in X_J^+$, let $L_J^\zeta(\lambda)$ be the simple integrable $\Uzlj$-module of highest weight $\lambda$. Define the fundamental alcove $C_\Z$ for $\Uz$ by replacing $p$ by $\ell$ in the definition for $C_\Z$ given in Section \ref{subsection:notation}.

The small quantum torus $\uzo \subset \uzg$ is a semisimple algebra, isomorphic to the group ring over $\C$ for the finite group $(\Z/\ell\Z)^n$. The set $X_\ell$ of $\ell$-restricted dominant weights is defined by $X_\ell = \set{\mu \in X^+:(\mu,\alpha^\vee) < \ell \text{ for all $\alpha \in \Delta$}}$. The irreducible $\uzo$-modules are parametrized by the set $X/\ell X$. Equivalently, the irreducible $\uzo$-modules are parametrized by the set $X_\ell$, which forms a set of coset representatives for $\ell X$ in $X$.

\subsection{Weight space decomposition}

Let $\lambda \in X^+$. Then $\Lz(\lambda)$ is by restriction a $\uzb$-module. Since $\uzb$ is flat as a right $\uzu$-module, and since $\uzu$ is normal in $\uzb$ with quotient $\uzb//\uzu \cong \uzo$, the cohomology space $\opH^\bullet(\uzu,\Lz(\lambda))$ is naturally a graded left $\uzo$-module. In this section we describe the $\uzo$-weight space decomposition of $\opH^\bullet(\uzu,\Lz(\lambda))$.

\begin{lemma} \label{lemma:admissibleweights}
Let $\lambda \in X^+$ and $\mu \in X$. Then $\Hom_{\uzo}(\mu,\opH^\bullet(\uzu,\Lz(\lambda))) = 0$ unless $\mu = w \cdot \lambda + \ell \sigma$ for some $w \in W$ and some $\sigma \in X$.
\end{lemma}

\begin{proof}
The subalgebra $\uzu$ acts trivially on the one-dimensional $\uzb$-module $-\mu$, so we get
\begin{align*}
\Hom_{\uzo}(\mu,\opH^\bullet(\uzu,\Lz(\lambda))) &\cong \Hom_{\uzo}(k,\opH^\bullet(\uzu,\Lz(\lambda)) \otimes -\mu) \\
&\cong \Hom_{\uzo}(k,\opH^\bullet(\uzu,\Lz(\lambda) \otimes -\mu)) \\
&\cong \opH^\bullet(\uzb,\Lz(\lambda) \otimes -\mu)
\end{align*}
The last isomorphism follows by applying the LHS spectral sequence for the algebra $\uzb$ and its normal subalgebra $\uzu$, and by using the fact that $\uzo$ is a semisimple algebra. Now given a $\uzb$-module $M$, let $Z'(M) := \uzg \otimes_{\uzb} M$ be the module obtained via tensor induction from $\uzb$ to $\uzg$. Then
\[
\opH^\bullet(\uzb,\Lz(\lambda) \otimes -\mu) \cong \Ext_{\uzb}^\bullet(\mu,\Lz(\lambda)) \cong \Ext_{\uzg}^\bullet(Z'(\mu),\Lz(\lambda)).
\]
By the Linkage Principle for $\uzg$, the last $\Ext$-group is zero unless $\mu = w \cdot \lambda + \ell \sigma$ for some $w \in W$ and some $\sigma \in X$. 
\end{proof}

\begin{lemma} \label{lemma:dotactionmodell}
Let $\lambda \in C_\Z$. Assume that $\ell$ is odd, that $\ell$ is coprime to $n+1$ if $\Phi$ is of type $A_n$, and that $\ell$ is coprime to $3$ if $\Phi$ is of type $E_6$ or $G_2$. Suppose $w_1 \cdot \lambda = w_2 \cdot \lambda + \ell \sigma$ for some $w_1,w_2 \in W$ and some $\sigma \in X$. Then $\sigma =0$ and $w_1 = w_2$.
\end{lemma}

\begin{proof}
Suppose $w_1 \cdot \lambda = w_2 \cdot \lambda + \ell \sigma$. Then
\[
(\lambda + \rho) - w_1^{-1}w_2(\lambda + \rho) = \ell(w_1^{-1}\sigma).
\]
Since $\lambda + \rho$ is a strongly dominant weight, the left-hand side of the above equation is a sum of positive roots by \cite[Lemma 13.2A]{Hum:1978}. In particular, $\ell(w_1^{-1}\sigma) \in \Z\Phi$. By assumption, $\ell$ does not divide the order of the finite group $X/(\Z\Phi)$, so $w_1^{-1} \sigma \in \Z\Phi$, and hence $\sigma \in \Z\Phi$. Then $\sigma = 0$ by Lemma \ref{lemma:dotactionmodp} (the lemma remains true if the prime $p$ is replaced by an arbitrary integer), and hence $w_1 = w_2$ by Lemma \ref{lemma:dotactionfaithful}.
\end{proof}

\begin{remark} \label{remark:dotactionmodell}
If $\lambda = 0$, then the conclusion to Lemma \ref{lemma:dotactionmodell} also holds if $\ell$ is odd, if $\ell > h$, and if $\ell$ is coprime to $3$ if $\Phi$ is of type $G_2$; for details see the argument in the proof of \cite[Theorem 2.5]{GK:1993}.
\end{remark}

\begin{corollary} \label{corollary:weightspacedecomposition}
Let $\lambda \in C_\Z$. Assume that $\ell$ is odd, that $\ell$ is coprime to $n+1$ if $\Phi$ has type $A_n$, and that $\ell$ is coprime to $3$ if $\Phi$ has type $E_6$ or $G_2$. Then
\[
\opH^\bullet(\uzu,\Lz(\lambda)) \cong \bigoplus_{w \in W} \Hom_{\uzo}(w \cdot \lambda,\opH^\bullet(\uzu,\Lz(\lambda))).
\]
\end{corollary}

\begin{proof}
The graded $\uzo$-module $\opH^\bullet(\uzu,\Lz(\lambda))$ decomposes as a direct sum of simple $\uzo$-modules. By a slight abuse of notation, we write this decomposition as
\[
\opH^\bullet(\uzu,\Lz(\lambda)) \cong \bigoplus_{\mu \in X/\ell X} \Hom_{\uzo}(\mu,\opH^\bullet(\uzu,\Lz(\lambda))).
\]
By Lemma \ref{lemma:admissibleweights}, the only non-zero summands are those for which $\mu \equiv w \cdot \lambda \mod \ell X$ for some $w \in W$, and by Lemma \ref{lemma:dotactionmodell}, the weights $w \cdot \lambda$ for $w \in W$ are all incongruent modulo $\ell X$.
\end{proof}

\begin{remark}\label{remark:weightspacedecomposition}
Using Remark \ref{remark:dotactionmodell}, the corollary also holds if $\lambda = 0$, if $\ell$ is odd, if $\ell > h$, and if $\ell$ is coprime to $3$ if $\Phi$ is of type $G_2$.
\end{remark}

\subsection{Kostant's theorem for quantum groups}

Let $\lambda \in C_\Z$. The cohomology calculations of Sections \ref{section:UJ1ringstructure} and \ref{section:paraboliccohomology} were critically dependent on Kostant's explicit formula for the $B$-module structure of the ordinary Lie algebra cohomology $\opH^\bullet(\fraku,L(\lambda))$. Equivalently, Kostant's formula computes the $\Dist(B)$-module structure of $\opH^\bullet(\U(\fraku),L(\lambda))$, where $\U(\fraku)$ is the universal enveloping algebra for $\fraku$.

In the root-of-unity quantum setting, the correct analogs for $\Dist(B)$ and $\U(\fraku)$ are $\Uz(\frakb)$ and $\cuzu$, respectively. Just as we needed an explicit description for the $\Dist(B)$-module structure on $\opH^\bullet(\U(\fraku),L(\lambda))$ to compute $\opH^\bullet(U_1,L(\lambda))$, so too will we need an explicit description of the $\Uz(\frakb)$-module structure on $\opH^\bullet(\cuzu,\Lz(\lambda))$ to compute $\opH^\bullet(\uzu,\Lz(\lambda))$. More generally, one can compute the $\Uzlj$-module structure of the cohomology space $\opH^\bullet(\cuzuj,\Lz(\lambda))$. This is essentially done already in \cite[\S 6]{UGA:2010}, though the setup there is not quite the same as what we need here (in \cite{UGA:2010}, the cohomology space $\opH^\bullet(\cuzuj,\Lz(\lambda))$ is computed as a $\cuzlj$-module instead of as a $\Uzlj$-module). In this section we make some brief remarks to explain how the results of \cite[\S 6]{UGA:2010} can be modified to suit our needs.

Let $\lambda \in X^+$. Then the simple integrable $\Uz$-module $\Lz(\lambda)$ is made a $\cuz$-module via the quotient $\cuz \twoheadrightarrow \uz$ and the inclusion $\uz \hookrightarrow \Uz$. By restriction, $\Lz(\lambda)$ is a module for the algebras $\Uzlj \subset \Uz$ and $\cuzuj \subset \cuz$. The right adjoint action of $\Uq$ on itself induces a right action of $\Uzpj$ on $\cuzuj$ \cite[\S 2.7]{Bendel:2011}. Then by \cite[Theorem 4.3.1]{Drupieski:2011}, the cohomology space $\opH^\bullet(\cuzuj,\Lz(\lambda))$ is naturally a graded left $\Uzpj$-module, and $\opH^\bullet(\cuzuj,\C)$ is a graded $\Uzpj$-module algebra. By restriction, $\opH^\bullet(\cuzuj,\Lz(\lambda))$ and $\opH^\bullet(\cuzuj,\C)$ are modules for the Levi subalgebra $\Uzlj \subset \Uzpj$.

The $\Uzlj$-module structure on $\opH^\bullet(\cuzuj,\Lz(\lambda))$ can be computed using the same arguments as in \cite[\S 6]{UGA:2010}, once two important changes are made. First, the algebra $\cuzlj$ should be replaced by the algebra $\Uzlj$. Second, the algebra $K$ defined in \cite[\S 6.2]{UGA:2010} should be replaced by the algebra $\Uzo$. Then the subalgebra $\Uzo \uzlj$ of $\Uz$ generated by $\Uz$ and $\uzlj$ is the correct quantum analog of the group scheme $(L_J)_1 T$. With these modifications, the main result of \cite[\S 6.4]{UGA:2010} may be stated as follows:

\begin{theorem} \label{theorem:quantumkostant} \textup{(cf.\ \cite[Theorem 6.4.1]{UGA:2010})}
Let $\lambda \in \CZbar$. Assume that $\ell$ is odd, that $\ell$ is coprime to $3$ if $\Phi$ has type $G_2$, and that $\ell \geq h-1$. Then for each $n \in \N$, there exists a $\Uzlj$-module isomorphism
\[
\opH^n(\cuzuj,\Lz(\lambda)) \cong \bigoplus_{\substack{w \in {}^JW \\ \ell(w) = n}} L_J^\zeta(w \cdot \lambda).
\]
\end{theorem}

\subsection{The ring structure of \texorpdfstring{$\opH^\bullet(\cuzu,\C)$}{H(Uzu,C)}}

In this section we provide a semi-explicit description for the ring structure on $\opH^\bullet(\cuzu,\C)$. The description is similar to that for $\opH^\bullet(\fraku,\C)$ given in Remark \ref{remark:ordinaryringstructure}, though significantly more work is required to obtain the ring structure in the quantum setting because of the lack of an explicit projective resolution for $\cuzu$.

Let $\{\gamma_1,\ldots,\gamma_N\}$ be an enumeration of $\Phi^+$ as in \cite[\S 9.3]{De-Concini:1993}, and let $E_{\gamma_1},\ldots,E_{\gamma_N} \in \cuzu$ be the corresponding positive root vectors. Set $\Lambda = \N^{N+1}$, viewed as a totally ordered semigroup via the reverse lexicographic order. Then by \cite[Theorem 9.3 and \S 10.1]{De-Concini:1993}, there exists a multiplicative $\Lambda$-filtration on $\cuzu$ such that the associated graded algebra $\gr^\Lambda \cuzu$ is generated by the symbols $\{E_{\gamma_1},\ldots,E_{\gamma_N}\}$ subject to the relations
\[
E_{\gamma_i}E_{\gamma_j} = \zeta^{(\gamma_i,\gamma_j)} E_{\gamma_j}E_{\gamma_i} \qquad \text{if $1 \leq i < j \leq N$.}
\]
Alternately, by \cite[Remark 10.1]{De-Concini:1993}, there exists a sequence of algebras $U^{(-1)},U^{(0)},U^{(1)},\ldots,U^{(N)}$ such that
\begin{enumerate}
\item $U^{(-1)} = \cuzu$;
\item For $0 \leq i \leq N$, $U^{(i-1)}$ admits a multiplicative $\N$-filtration;
\item For $1 \leq i \leq N$, $U^{(i)}$ is the associated graded algebra of $U^{(i-1)}$; and
\item $U^{(N)} \cong \gr^\Lambda \cuzu$.
\end{enumerate}
It follows that for each $1 \leq i \leq N$, there exists a spectral sequence of algebras satisfying
\begin{equation} \label{eq:Mayspecseq}
E_1^{p,q} = \opH^{p+q}(U^{(i)},\C)_{(p)} \Rightarrow \opH^{p+q}(U^{(i-1)},\C),
\end{equation}
where the subscript on the $E_1$-term denotes the internal grading on $\opH^\bullet(U^{(i)},\C)$ arising from the $\N$-grading on $U^{(i)}$. The right adjoint action of $\Uzo$ on $\cuzu$ passes to an action of $\Uzo$ on each $U^{(i)}$, so \eqref{eq:Mayspecseq} is also a spectral sequence of $\Uzo$-module algebras.

By \cite[Proposition 2.1]{GK:1993}, the cohomology ring $\opH^\bullet(U^{(N)},\C)$ is isomorphic to the graded ring $\Lambda_\zeta^\bullet$ generated by the symbols $\{x_{\gamma_1},\ldots,x_{\gamma_N}\}$ (each of graded degree one), subject to the relations
\begin{align*}
x_{\gamma_i}x_{\gamma_j} + \zeta^{-(\gamma_i,\gamma_j)}x_{\gamma_j}x_{\gamma_i} &= 0 & \text{if $1 \leq i < j \leq N$, and} \\
x_{\gamma_i}^2 &= 0 &\text{for all $1 \leq i \leq N$.}
\end{align*}
The $\Uzo$-module structure on $\Lambda_\zeta^\bullet$ is obtained by assigning $x_{\gamma_i}$ to have weight $-\gamma_i$, so there exists an equality of formal characters $\ch \Lambda_\zeta^\bullet = \ch \Lambda^\bullet(\fraku^*)$. In particular, for each $w \in W$, the $(w \cdot 0)$-weight space of $\Lambda_\zeta^\bullet$ (and hence also of $\opH^\bullet(U^{(i)},\C)$ for each $-1 \leq i \leq N$) is one-dimensional.

Given $w \in W$ with $\Phi(w) = \set{\beta_1,\ldots,\beta_n}$ ordered as in Remark \ref{remark:ordinaryringstructure}, set $f_{\Phi(w)} = x_{\beta_1} \cdots x_{\beta_n} \in \Lambda_\zeta^{\bullet}$. Since $\opH^\bullet(\cuzu,\C)$ is spanned by weight vectors of weights $w \cdot 0$ with $w \in W$ (by Theorem \ref{theorem:quantumkostant}), and since for each $w \in W$ the $(w \cdot 0)$-weight space of $\Lambda_\zeta^\bullet$ is one-dimensional, the (image of the) vector $f_{\Phi(w)}$ must for all $1 \leq i \leq N$ be a permanent cycle in the spectral sequence \eqref{eq:Mayspecseq}, and (the image of) the set $\{f_{\Phi(w)}: w \in W \}$ must form a basis for $\opH^\bullet(\cuzu,\C)$. Finally, the uni-dimensionality of the $(w \cdot 0)$-weight spaces in $\Lambda_\zeta^\bullet$ implies that the cup product between weight vectors in $\opH^\bullet(\cuzu,\C)$ is induced by the multiplication in $\Lambda_\zeta^\bullet$.\footnote{In a spectral sequence of algebras, the product on the $E_\infty$-page will determine the product on the abutment up to terms in lower filtered degree. In our situation, the product also respects weight spaces, and the unidimensionality of certain weight spaces implies that there are no vectors of the correct weight in lower filtered degree.}

\subsection{Weight space structure}

Now we look at the structure of the $\uzo$-weight space
\[
\Hom_{\uzo}(w \cdot \lambda,\opH^\bullet(\uzu,\Lz(\lambda))) \cong \opH^\bullet(\uzb,\Lz(\lambda) \otimes -w \cdot \lambda)
\]
as a left module for $\opH^\bullet(\uzb,\C)$ under the cup product. Note that $\uzu$ is a left coideal subalgebra in $\uzb$ (i.e., the coproduct $\Delta$ on $\uzb$ satisfies $\Delta(\uzu) \subset \uzb \otimes \uzu$), so it makes sense to consider the left cup product action
\[
\cup: \opH^\bullet(\uzb,\C) \otimes \opH^\bullet(\uzu,\Lz(\lambda) \otimes -w \cdot \lambda) \rightarrow \opH^\bullet(\uzu,\Lz(\lambda) \otimes -w \cdot \lambda).
\]

Let $\set{x_{\gamma_1},\ldots,x_{\gamma_N}} \subset \fraku$ and $\set{f_{\gamma_1},\ldots,f_{\gamma_N}} \subset \fraku^*$ be dual bases for $\fraku$ and $\fraku^*$ as in Remark \ref{remark:ordinaryringstructure}. For $0 \leq i \leq N$, let $R_i \subset \fraku^*$ be the span of the set $\{f_{\gamma_1},\ldots,f_{\gamma_i}\}$. Then $S^\bullet(R_N) \cong S^\bullet(\fraku^*)$. Let $\calZ_i$ be the subalgebra of $\cuzb$ generated by the set $\{E_{\gamma_1}^\ell,\ldots,E_{\gamma_i}^\ell\}$. Then $\calZ_i$ is normal in $\cuzb$ because the generators for $\calZ$ are central in $\cuz$. Set $A_i = \cuzb//\calZ_i$. Then $A_0 = \cuzb$ and $A_N = \cuzb//\calZ^+ \cong \uzb$. Let $B_i \subset A_{i-1}$ be the (normal) subalgebra generated by $E_{\gamma_i}^\ell$. Then $A_{i-1}//B_i \cong A_i$. Moreover, it follows from the description of the coproduct in Lemma \ref{lemma:ellthpowerscoproducts} that the algebras $A_i$ and $B_i$ inherit bialgebra structures from $\cuzb$.

\begin{proposition} \label{proposition:weightspaceinduction}
Assume that $\ell$ is odd, that $\ell > h$, and that $\ell$ is coprime to $3$ if $\Phi$ is of type $G_2$. Then for all $0 \leq i \leq N$, $\opH^{\textup{odd}}(A_i,\C) = 0$ and $\opH^{2\bullet}(A_i,\C) \cong S^\bullet(R_i)^{(1)}$ as $\Uzo$-module algebras. Now let $\lambda \in C_\Z$ and $w \in W$, and suppose that also $\ell$ is coprime to $n+1$ if $\Phi$ is of type $A_n$, and that $\ell$ is coprime to $3$ if $\Phi$ is of type $E_6$. Then the space $\opH^\bullet(A_i,\Lz(\lambda) \otimes -w \cdot \lambda)$ is free as a left $\opH^\bullet(A_i,\C)$-module under the cup product, generated by a vector in degree $\ell(w)$ of $\Uzo$-weight zero.
\end{proposition}

\begin{proof}
The proof is by induction on $i$. For $i=0$, we have, as in the proof of Lemma \ref{lemma:admissibleweights},
\[
\opH^\bullet(\cuzb,\Lz(\lambda) \otimes -w \cdot \lambda) \cong ( \opH^\bullet(\cuzu,\Lz(\lambda)) \otimes -w \cdot \lambda )^{\uzo}.
\]
By Theorem \ref{theorem:quantumkostant}, $\opH^\bullet(\cuzu,\Lz(\lambda))$ decomposes as a direct sum of one-dimensional $\Uzo$-modules
\[
\opH^\bullet(\cuzu,\Lz(\lambda)) \cong \bigoplus_{w' \in W} \C_{w' \cdot \lambda},
\]
with the summand $\C_{w' \cdot \lambda}$ appearing in degree $\ell(w')$. Since the $w' \cdot \lambda$ are all distinct as weights for $\uzo$ by Lemma \ref{lemma:dotactionmodell} (cf.\ also Remark \ref{remark:dotactionmodell}), we get
\[
\opH^\bullet(\cuzb,\Lz(\lambda) \otimes -w \cdot \lambda) = \opH^{\ell(w)}(\cuzb,\Lz(\lambda) \otimes -w \cdot \lambda) \cong \C.
\]
In particular, if $\lambda = 0$ and $w=1$, then $\Lz(\lambda) \otimes -w \cdot \lambda = \C$, and we get $\opH^\bullet(\cuzb,\C) \cong \C$. This establishes the induction hypothesis.

Now let $0 < i \leq N$, and assume that the proposition is true for $i-1$. Set $V = \Lz(\lambda) \otimes -w \cdot \lambda$. Since $A_{i-1}$ is free over its central subalgebra $B_i$, we can consider the pair of LHS spectral sequences
\begin{align}
E_2^{a,b}(\C) &= \opH^a(A_{i-1}//B_i,\opH^b(B_i,\C)) \Rightarrow \opH^{a+b}(A_{i-1},\C), \quad \text{and} \label{eq:Cinductionspecseq} \\
E_2^{a,b}(V) &= \opH^a(A_{i-1}//B_i,\opH^b(B_i,V)) \Rightarrow \opH^{a+b}(A_{i-1},V). \label{eq:Vinductionspecseq}
\end{align}
Because $B_i \subset A_{i-1}$ are bialgebras, \eqref{eq:Cinductionspecseq} is a spectral sequence of algebras, and \eqref{eq:Vinductionspecseq} is a module over \eqref{eq:Cinductionspecseq}. The algebra $B_i$ acts trivially on $V$, so $\opH^\bullet(B_i,V) \cong V \otimes \opH^\bullet(B_i,\C)$ as an $A_{i-1}//B_i$-module. Moreoever, since $B_i$ is central in $A_{i-1}$, the action of $A_{i-1}//B_i \cong A_i$ on $\opH^\bullet(B_i,\C)$ is trivial by \cite[Lemma 5.2.2]{GK:1993}. Then the spectral sequences can be rewritten as
\begin{align}
E_2^{a,b}(\C) &= \opH^b(B_i,\C) \otimes \opH^a(A_i,\C) \Rightarrow \opH^{a+b}(A_{i-1},\C), \quad \text{and} \label{eq:rewriteCinductionspecseq} \\
E_2^{a,b}(V) &= \opH^b(B_i,\C) \otimes \opH^a(A_i,V) \Rightarrow \opH^{a+b}(A_{i-1},V). \label{eq:rewriteVinductionspecseq}
\end{align}
The identification between the $E_2$-pages of \eqref{eq:Cinductionspecseq} and \eqref{eq:rewriteCinductionspecseq} is an isomorphism of graded-com\-mutative algebras.

From now on, write $E_2^{a,b}$ to denote generically a term in the $E_2$-page of either \eqref{eq:rewriteCinductionspecseq} or \eqref{eq:rewriteVinductionspecseq}. Since $B_i$ is a polynomial algebra on a generator of weight $\ell\gamma_i$ for $\Uzo$, $\opH^\bullet(B_i,\C)$ is an exterior algebra on the one-dimensional vector space $\C_i$ of weight $-\ell \gamma_i$, i.e., $\opH^\bullet(B_i,\C) \cong \Lambda^\bullet(\C_i)$. Then $E_2^{a,1} \cong \C_i \otimes E_2^{a,0}$, $E_2^{a,b} = 0$ for all $b \geq 2$, and the only possible non-zero differentials have the form $d_2: E_2^{a,1} \rightarrow E_2^{a+2,0}$. Furthermore, arguing exactly as in \cite[\S 5.3]{Bendel:2011}, one can use the induction hypothesis to show that $E_2^{a,0} = 0$ for $a < \ell(w)$, that $E_2^{\ell(w),0} \cong \C$, that $E_2^{\ell(w)+a,0} = 0$ for all odd $a$, and that $d_2: E_2^{\ell(w)+a,1} \rightarrow E_2^{\ell(w)+a+2,0}$ is injective for all even $a$ (taking $w=1$ for \eqref{eq:rewriteCinductionspecseq}). These observations imply that $E_\infty^{a,b} = 0$ if $b \geq 1$, and hence that the edge maps $\opH^\bullet(A_i,\C) \rightarrow \opH^\bullet(A_{i-1},\C)$ and $\opH^\bullet(A_i,V) \rightarrow \opH^\bullet(A_{i-1},V)$ of \eqref{eq:rewriteCinductionspecseq} and \eqref{eq:rewriteVinductionspecseq} are surjective.

Fix $0 \neq v \in \C_i$, and set $z_i = d_2(v) \in \opH^2(A_i,\C)$. The vector $z_i$ is central in $\opH^\bullet(A_i,\C)$, because the cohomology ring of a bialgebra is always graded-commutative. Now fix elements $z_1,\ldots,z_{i-1} \in \opH^2(A_i,\C)$ of weights $-\ell\gamma_1,\ldots,-\ell\gamma_{i-1}$ lifting the polynomial generators for $\opH^\bullet(A_{i-1},\C)$. For $r \geq 0$, let $S_i^r \subseteq \opH^{2r}(A_i,\C)$ be the subspace spanned by all homogenous degree-$r$ monomials in the elements $z_1,\ldots,z_i$. Then $S_i^r$ is a quotient of the space $S^r(R_i)$ defined prior to the statement of the theorem, with the quotient map $S^r(R_i) \rightarrow S_i^r$ defined by $f_j \mapsto z_j$. Moreover, there exists a natural multiplication map $S^r(R_i) \otimes \opH^{\ell(w)}(A_i,V) \rightarrow \opH^{\ell(w)+2r}(A_i,V)$ induced by the composition of the quotient map $S^r(R_i) \rightarrow S_i^r$ with the natural action of $S_i^r \subseteq \opH^{2r}(A_i,\C)$ on $\opH^{\ell(w)}(A_i,V)$.

We claim that the multiplication map $S^r(R_i) \otimes \opH^{\ell(w)}(A_i,V) \rightarrow \opH^{\ell(w)+2r}(A_i,V)$ is a bijection. Indeed, taking $\lambda = 0$ and $w=1$, then the claim is equivalent to showing for all $r \geq 0$ that $S^r(R_i) \cong \opH^{2r}(A_i,\C)$, and hence that $\Hbul(A_i,\C)$ is a polynomial algebra generated by the elements $\set{z_1,\ldots,z_i}$. Then taking arbitrary $\lambda \in C_\Z$ and $w \in W$, the claim implies that $\opH^\bullet(A_i,V)$ is generated freely as a $\opH^\bullet(A_i,\C)$-module over the one-dimensional space $\opH^{\ell(w)}(A_i,V)$.

We proceed to prove the claim. For $r=0$ the claim is true by the observation $E_2^{\ell(w),0} \cong \C$, so let $r \geq 1$ and assume that the claim is true for $r-1$. Then there exists a short exact sequence
\[
0 \longrightarrow E_2^{\ell(w)+2(r-1),1} \stackrel{d_2}{\longrightarrow} E_2^{\ell(w)+2r,0} \longrightarrow E_\infty^{\ell(w)+2r,0} \longrightarrow 0,
\]
which may be rewritten as
\begin{equation} \label{eq:SES}
0 \longrightarrow \C_i \otimes \opH^{\ell(w)+2(r-1)}(A_i,V) \stackrel{d_2}{\longrightarrow} \opH^{\ell(w)+2r}(A_i,V) \longrightarrow \opH^{\ell(w)+2r}(A_{i-1},V) \longrightarrow 0.
\end{equation}
By induction on $i$ and $r$, the natural multiplication maps induce isomorphisms
\begin{gather*}
S^{r-1}(R_i) \otimes \opH^{\ell(w)}(A_i,V) \cong \opH^{\ell(w)+2(r-1)}(A_i,V) \quad \text{and} \\
S^r(R_{i-1}) \otimes \opH^{\ell(w)}(A_{i-1},V) \cong \opH^{\ell(w)+2r}(A_{i-1},V).
\end{gather*}
Moreover, since \eqref{eq:rewriteVinductionspecseq} is a module over \eqref{eq:rewriteCinductionspecseq}, there exists a commutative diagram
\[
\xymatrix{
S^r(R_{i-1}) \otimes \opH^{\ell(w)}(A_i,V) \ar@{->}[r] \ar@{->}[d]^{\sim} & \opH^{\ell(w)+2r}(A_i,V) \ar@{->}[d] \\
S^r(R_{i-1}) \otimes \opH^{\ell(w)}(A_{i-1},V) \ar@{->}[r]^(.57){\sim} & \opH^{\ell(w)+2r}(A_{i-1},V)
}
\]
where the vertical maps are induced by the edge maps of \eqref{eq:rewriteVinductionspecseq}, and the horizontal maps are the natural multiplication maps. The left-hand vertical map in the diagram is an isomorphism because the edge map $E_2^{\ell(w),0} \rightarrow E_\infty^{\ell(w),0}$ is surjective and both spaces are one-dimensional. Also using the fact that \eqref{eq:rewriteCinductionspecseq} is a spectral sequence of algebras and that \eqref{eq:rewriteVinductionspecseq} is a module over \eqref{eq:rewriteCinductionspecseq}, we see that the differential $d_2: \C_i \otimes E_2^{a,0} \rightarrow E_2^{a+2,0}$ is just multplication on $E_2^{a+2,0}$ by $f_i$. Combining the above observations, we conclude that the multiplication map $S^r(R_i) \otimes \opH^{\ell(w)}(A_i,V) \rightarrow \opH^{\ell(w)+2r}(A_i,V)$ is surjective, and also injective by dimension count, and hence an isomorphism.
\end{proof}

\begin{corollary} \label{corollary:freeweightspace}
Assume that $\ell$ is odd, that $\ell > h$, and that $\ell$ is coprime to $3$ if $\Phi$ has type $G_2$. Then $\opH^{\textup{odd}}(\uzb,\C) = 0$ and $\opH^{2\bullet}(\uzb,\C) \cong S^\bullet(\fraku^*)^{(1)}$ as $\Uzo$-module algebras. Now let $\lambda \in C_\Z$ and $w \in W$, and suppose that also $\ell$ is coprime to $n+1$ if $\Phi$ is of type $A_n$, and that $\ell$ is coprime to $3$ if $\Phi$ is of type $E_6$. Then $\opH^\bullet(\uzb,\Lz(\lambda) \otimes -w \cdot \lambda)$ is free as a left $\opH^\bullet(\uzb,\C)$-module under the cup product, generated by a vector in degree $\ell(w)$ of weight zero for $\Uzo$.
\end{corollary}

\begin{proof}
This is the case $i=N$ of the proposition.
\end{proof}

\subsection{Cohomology for the nilpotent small quantum group}

Now we are ready to compute the structure of the cohomology space $\opH^\bullet(\uzu,\Lz(\lambda))$.

\begin{theorem} \label{theorem:uzucohomology}
Assume that $\ell$ is odd, that $\ell > h$, that $\ell$ is coprime to $n+1$ if $\Phi$ has type $A_n$, and that $\ell$ is coprime to $3$ if $\Phi$ has type $E_6$ or $G_2$. Let $\lambda \in C_\Z$ and $w \in W$. Then there exists an isomorphism of left $\Uzo$-modules and of left $\opH^\bullet(\uzb,\C)$-modules
\begin{equation} \label{eq:uzuiso}
\opH^\bullet(\uzu,\Lz(\lambda)) \cong \opH^\bullet(\uzb,\C) \otimes \opH^\bullet(\cuzu,\Lz(\lambda)),
\end{equation}
with $\opH^\bullet(\uzb,\C)$ acting via the cup product on $\opH^\bullet(\uzu,\Lz(\lambda))$, and via left multiplication on $\opH^\bullet(\uzb,\C) \otimes \opH^\bullet(\cuzu,\Lz(\lambda))$.
\end{theorem}

\begin{proof}
The theorem follows by applying Corollaries \ref{corollary:weightspacedecomposition} and \ref{corollary:freeweightspace}.
\end{proof}

The right adjoint action of $\Uzb$ on itself stabilizes the subspace $\uzu$ \cite[\S 2.7]{Bendel:2011}. By \cite[Theorem 4.3.1]{Drupieski:2011}, this induces a left action of $\Uzb$ on the cohomology space $\opH^\bullet(\uzu,\Lz(\lambda))$.

\begin{theorem} \label{theorem:Uzbiso}
The isomorphism \eqref{eq:uzuiso} is an isomorphism of left $\Uzb$-modules.
\end{theorem}

\begin{proof}
The subalgebra $\uzu \subset \Uzb$ acts trivially on $\opH^\bullet(\uzu,\Lz(\lambda))$, while the root vectors $\{E_\gamma^{(n\ell)}: \gamma \in \Phi^+, n \geq 1 \}$ all commute with the subalgebra $\uzo \subset \Uzb$. It follows then that the action of $\Uzb$ on $\opH^\bullet(\uzu,\Lz(\lambda))$ leaves stable the $\uzo$-weight spaces, so it suffices to show for each $w \in W$ that the $\Uzo$-module isomorphism
\[
\Hom_{\uzo}(w \cdot \lambda,\opH^\bullet(\uzu,\Lz(\lambda))) \cong \opH^\bullet(\uzb,\Lz(\lambda) \otimes -w \cdot \lambda) \cong \opH^{\bullet-\ell(w)}(\uzb,\C)
\]
is an isomorphism of $\Uzb$-modules.

Set $V = \Lz(\lambda) \otimes -w \cdot \lambda$. The theorem now follows from two observations. First, $\opH^{\ell(w)}(\uzb,V)$ is by Corollary \ref{corollary:freeweightspace} a one-dimensional $\Uzb$-module of $\Uzo$-weight zero, that is, is isomorphic to the trivial module for $\Uzb$. Second, the cup product
\begin{equation} \label{eq:backwardscupproduct}
\cup: \opH^{\ell(w)}(\uzb,V) \otimes \opH^{\bullet}(\uzb,\C) \rightarrow \opH^{\bullet+\ell(w)}(\uzb,V)
\end{equation}
is a homomorphism of $\Uzb$-modules. To see this, use the definition for the adjoint actions of $\Uzb$ on $\opH^\bullet(\uzb,\C)$ and $\opH^\bullet(\uzb,V)$ given in \cite[(4.3.1)]{Drupieski:2011}, together with the explicit description for the cup product at the level of cocycles given in \cite[(5.3)]{Drupieski:2011a}. But the cup product in \eqref{eq:backwardscupproduct} is equivalent to the cup product
\[
\cup: \opH^\bullet(\uzb,\C) \otimes \opH^{\ell(w)}(\uzb,V) \rightarrow \opH^{\bullet+\ell(w)}(\uzb,V)
\]
by \cite[VIII.4]{Mac-Lane:1995}, which is an isomorphism by Corollary \ref{corollary:freeweightspace}.
\end{proof}

\subsection{Ring structure}

Though $\uzu$ is not a Hopf algebra, $\opH^\bullet(\uzu,\C) = \Ext_{\uzu}^\bullet(\C,\C)$ is still a ring under the Yoneda composition of extensions.

\begin{theorem}
Assume that $\ell$ is odd, that $\ell > 2(h-1)$, that $\ell$ is coprime to $n+1$ if $\Phi$ is of type $A_n$, and that $\ell$ is coprime to $3$ if $\Phi$ is of type $E_6$ or $G_2$. Then there exists a graded ring isomorphism
\[
\opH^\bullet(\uzu,\C) \cong \opH^\bullet(\uzb,\C) \otimes \opH^\bullet(\cuzu,\C).
\]
\end{theorem}

\begin{proof}
By Theorem \ref{theorem:uzucohomology}, there exists an isomorphism of $\Uzo$- and $\Hbul(\uzb,\C)$-modules
\begin{equation} \label{eq:uzudecomposition}
\opH^\bullet(\uzu,\C) \cong \opH^\bullet(\uzb,\C) \otimes \opH^\bullet(\cuzu,\C),
\end{equation}
Also, $\Hbul(\uzb,\C) \cong S^{\bullet/2}(\fraku^*)^{(1)}$ as $\Uzo$-modules by Corollary \ref{corollary:freeweightspace}. Consider the LHS spectral sequence of $\Uzo$-modules:
\[
E_2^{a,b} = \opH^a(\cuzu//\calZ^+,\opH^b(\calZ^+,\C)) \Rightarrow \opH^{a+b}(\cuzu,\C).
\]
Since $\calZ^+$ is central in $\cuzu$ and $\cuzu//\calZ^+ \cong \uzu$, the spectral sequence may be rewritten as
\begin{equation} \label{eq:LHSspecseqdecomposition}
E_2^{a,b} = \opH^b(\calZ^+,\C) \otimes \opH^a(\uzu,\C) \Rightarrow \opH^{a+b}(\cuzu,\C).
\end{equation}
Moreoever, $\opH^\bullet(\calZ^+,\C) \cong \Lambda^\bullet(\fraku^*)^{(1)}$ as a $\Uzo$-module. We claim that the edge map $\opH^\bullet(\uzu,\C) \rightarrow \opH^\bullet(\cuzu,\C)$ of \eqref{eq:LHSspecseqdecomposition} is surjective. Indeed, by \cite[Theorem 6.4.1]{UGA:2010}, $\opH^\bullet(\cuzu,\C)$ decomposes as a direct sum of one-dimensional $\Uzo$-modules, $\opH^\bullet(\cuzu,\C) \cong \bigoplus_{w \in W} \C_{w \cdot 0}$. Fix $w \in W$, and suppose $w \cdot 0$ occurs as a weight of $\Uzo$ in $E_2^{a,b}$. Then by \eqref{eq:uzudecomposition} and \eqref{eq:LHSspecseqdecomposition}, $w \cdot 0 = w' \cdot 0 + \ell(\sigma_1 + \sigma_2)$ for some $w' \in W$ and some weights $\sigma_1$ of $\Lambda^\bullet(\fraku^*)$ and $\sigma_2$ of $S^\bullet(\fraku^*)$. In particular, $\sigma_1,\sigma_2 \in \N \Phi^-$. Then Lemma \ref{lemma:dotactionmodell} implies that $\sigma_1 + \sigma_2 = 0$, and hence that $\sigma_1 = \sigma_2 = 0$. This is only possible if $b=0$, so it follows that the edge map $\opH^\bullet(\uzu,\C) \rightarrow \opH^\bullet(\cuzu,\C)$ is surjective, and that under the $\Uzo$-module isomorphism of \eqref{eq:uzudecomposition}, the horizontal edge map in \eqref{eq:LHSspecseqdecomposition} is the projection onto the subspace
\[ 1 \otimes \opH^\bullet(\cuzu,\C) \subset \opH^\bullet(\uzb,\C) \otimes \opH^\bullet(\cuzu,\C).
\]

Now let $\calB \subset \opH^\bullet(\uzu,\C)$ be the $\Uzo$-submodule corresponding to the subspace $1 \otimes \opH^\bullet(\cuzu,\C)$ under the isomorphism \eqref{eq:uzudecomposition}. Considering the weights of $\calB$ and arguing as in the proof of Theorem \ref{theorem:U1cohomologyringstructure}, we see that $\calB$ is in fact a subalgebra of $\opH^\bullet(\uzu,\C)$. (In particular, Lemma \ref{lemma:sumofdotactionmodp} remains valid with the same proof if $\ell$ is substituted for $p$.) Since the edge map $\opH^\bullet(\uzu,\C) \rightarrow \opH^\bullet(\cuzu,\C)$ maps $\calB$ isomorphically onto $\opH^\bullet(\cuzu,\C)$, we conclude that $\calB \cong \opH^\bullet(\cuzu,\C)$ as algebras.

Recall that $\opH^\bullet(\uzb,\C)$ identifies with the subalgebra $\opH^\bullet(\uzu,\C)^{\uzo}$ of $\opH^\bullet(\uzu,\C)$. The inclusion $\opH^\bullet(\uzb,\C) \hookrightarrow \opH^\bullet(\uzu,\C)$ is just the restriction map in cohomology. Then from \eqref{eq:uzudecomposition} we see that $\opH^\bullet(\uzu,\C)$ is generated as an algebra by $\calB$ together with $\opH^\bullet(\uzb,\C)$. Finally, it follows from \cite[VIII.4]{Mac-Lane:1995} and the fact that $\opH^{\textup{odd}}(\uzb,\C) = 0$ that $\opH^\bullet(\uzb,\C)$ is a central subalgebra of $\opH^\bullet(\uzu,\C)$. We conclude that multiplication in $\opH^\bullet(\uzu,\C)$ induces an isomorphism of algebras
\[
\opH^\bullet(\uzb,\C) \otimes \opH^\bullet(\cuzu,\C) \cong \opH^\bullet(\uzu,\C)^{\uzo} \otimes \calB \stackrel{\sim}{\longrightarrow} \opH^\bullet(\uzu,\C). \qedhere
\]
\end{proof}

\subsection{Parabolic computations}

For completeness, we now state the quantum analog of the parabolic computations in Section \ref{section:paraboliccohomology}. Before stating the main result, we point out that the cohomology space $\opH^\bullet(\uzpj,\Lz(\lambda))$ is naturally a rational $\Dist(P_J)$-module (equivalently, a rational $P_J$-module). The $\Dist(P_J)$-module structure arises as follows. First, the right adjoint action of $\Uzpj$ on itself stabilizes the subalgebra $\uzpj \subset \Uzpj$. This gives rise to a natural action of $\Uzpj$ on $\opH^\bullet(\uzpj,\Lz(\lambda))$, which factors through the quotient $\Dist(P_J) \cong \Uzpj//\uzpj$.

\begin{theorem}
Assume that $\ell$ is odd, that $\ell$ is coprime to $n+1$ if $\Phi$ has type $A_n$, and that $\ell$ is coprime to $3$ if $\Phi$ has type $E_6$ or $G_2$. Let $\lambda \in X^+ \cap \CZbar$ (so $\ell > h$), and let $J \subseteq \Delta$.
\begin{enumalph}
\item If $\opH^\bullet(\uzpj,\Lz(\lambda)) = 0$ unless $\lambda$ is weakly $\ell$-linked to zero.
\item If $\lambda = w \cdot 0 + \ell \sigma$ for some $w \in W$ and $\sigma \in X$, then $\lambda = 0$ or $\sigma$ is minuscule.
\item Suppose $\lambda = w \cdot 0 + \ell \sigma$. Then there exists a $P_J$-module isomorphism
\[
\opH^j(\uzpj,\Lz(\lambda)) \cong
\begin{cases}
\ind_B^{P_J} [ S^{\frac{j-\ell(w)}{2}}(\fraku^*)\otimes w^{-1}\sigma] & \text{if } j\equiv l(w)\ \text{mod }2, \\
0 & \text{otherwise.}
\end{cases}
\]
\end{enumalph}
\end{theorem}

\begin{proof}
The proof is analogous to that given for Theorem~\ref{theorem:alggrpparabolic}, so the details are omitted.
\end{proof}

\section*{Acknowledgments}

The authors gratefully acknowledge the assistance of Jon Carlson in completing computer calculations in MAGMA for this project.


\begin{thebibliography}{BNPP}

\bibitem[AJS]{AJS:1994}
H.~H. Andersen, J.~C. Jantzen, and W.~Soergel, {\em Representations of quantum
  groups at a {$p$}th root of unity and of semisimple groups in characteristic
  {$p$}: independence of {$p$}}, Ast\'erisque (1994), no.~220, 321.

\bibitem[AJ]{AJ:1984}
H.~H. Andersen and J.~C. Jantzen, {\em Cohomology of induced representations
  for algebraic groups}, Math. Ann. \textbf{269} (1984), no.~4, 487--525.

\bibitem[BNPP]{Bendel:2011}
C.~P. Bendel, D.~K. Nakano, B.~J. Parshall, and C.~Pillen, {\em Cohomology for
  quantum groups via the geometry of the nullcone}, 2011, \href
  {http://arxiv.org/abs/1102.3639} {\path{arXiv:1102.3639}}.

\bibitem[BCP]{BCP:1997}
W.~Bosma, J.~Cannon, and C.~Playoust,
  \href{http://dx.doi.org/10.1006/jsco.1996.0125} {{\em The {M}agma algebra
  system. {I}. {T}he user language}}, J. Symbolic Comput. \textbf{24} (1997),
  no.~3-4, 235--265, Computational algebra and number theory (London, 1993).

\bibitem[CPS]{Cline:1983}
E.~Cline, B.~Parshall, and L.~Scott, {\em A {M}ackey imprimitivity theory for
  algebraic groups}, Math. Z. \textbf{182} (1983), no.~4, 447--471.

\bibitem[Cra]{Crane:1983}
R.~Crane, {\em Cohomology rings of infinitesimal unipotent groups}, Ph.D.
  thesis, University of Virginia, 1983.

\bibitem[DCP]{De-Concini:1993}
C.~De~Concini and C.~Procesi, {\em Quantum groups}, {$D$}-modules,
  representation theory, and quantum groups ({V}enice, 1992), Lecture Notes in
  Math., vol. 1565, Springer, Berlin, 1993, pp.~31--140.

\bibitem[DCK]{De-Concini:1990}
C.~De~Concini and V.~G. Kac, {\em Representations of quantum groups at roots of
  {$1$}}, Operator algebras, unitary representations, enveloping algebras, and
  invariant theory ({P}aris, 1989), Progr. Math., vol.~92, Birkh\"auser Boston,
  Boston, MA, 1990, pp.~471--506.

\bibitem[Dru1]{Drupieski:2011a}
C.~M. Drupieski, \href{http://dx.doi.org/10.1093/imrn/rnq156} {{\em On
  injective modules and support varieties for the small quantum group}}, Int.
  Math. Res. Not. \textbf{2011} (2011), no.~10, 2263--2294.

\bibitem[Dru2]{Drupieski:2011}
C.~M. Drupieski, \href{http://dx.doi.org/10.1016/j.jpaa.2010.09.006} {{\em
  Representations and cohomology for {F}robenius-{L}usztig kernels}}, J. Pure
  Appl. Algebra \textbf{215} (2011), no.~6, 1473--1491.

\bibitem[FP1]{FP:1986}
E.~M. Friedlander and B.~J. Parshall, {\em Cohomology of infinitesimal and
  discrete groups}, Math. Ann. \textbf{273} (1986), no.~3, 353--374.

\bibitem[FP2]{FP:1986a}
\bysame, {\em Cohomology of {L}ie algebras and algebraic groups}, Amer. J.
  Math. \textbf{108} (1986), no.~1, 235--253 (1986).

\bibitem[GK]{GK:1993}
V.~Ginzburg and S.~Kumar, {\em Cohomology of quantum groups at roots of unity},
  Duke Math. J. \textbf{69} (1993), no.~1, 179--198.

\bibitem[Hum]{Hum:1978}
J.~E. Humphreys, {\em Introduction to {L}ie algebras and representation
  theory}, Graduate Texts in Mathematics, vol.~9, Springer-Verlag, New York,
  1978, Second printing, revised.

\bibitem[Jan1]{Jantzen:1996}
J.~C. Jantzen, {\em Lectures on quantum groups}, Graduate Studies in
  Mathematics, vol.~6, American Mathematical Society, Providence, RI, 1996.

\bibitem[Jan2]{Jan:2003}
\bysame, {\em Representations of algebraic groups}, second ed., Mathematical
  Surveys and Monographs, vol. 107, American Mathematical Society, Providence,
  RI, 2003.

\bibitem[KLT]{KLT:1999}
S.~Kumar, N.~Lauritzen, and J.~F. Thomsen, {\em Frobenius splitting of
  cotangent bundles of flag varieties}, Invent. Math. \textbf{136} (1999),
  no.~3, 603--621.

\bibitem[Lus]{Lusztig:1990}
G.~Lusztig, {\em Quantum groups at roots of {$1$}}, Geom. Dedicata \textbf{35}
  (1990), no.~1-3, 89--113.

\bibitem[ML]{Mac-Lane:1995}
S.~Mac~Lane, {\em Homology}, Classics in Mathematics, Springer-Verlag, Berlin,
  1995, Reprint of the 1975 edition.

\bibitem[PT]{PT:2002}
P.~Polo and J.~Tilouine, {\em Bernstein-{G}elfand-{G}elfand complexes and
  cohomology of nilpotent groups over {$\mathbb Z_{(p)}$} for representations
  with {$p$}-small weights}, Ast\'erisque (2002), no.~280, 97--135, Cohomology
  of Siegel varieties.

\bibitem[UGA1]{UGA:2009}
U.~of~Georgia VIGRE Algebra~Group, {\em On {K}ostant's theorem for {L}ie
  algebra cohomology}, Representation theory, Contemp. Math., vol. 478, Amer.
  Math. Soc., Providence, RI, 2009, pp.~39--60.

\bibitem[UGA2]{UGA:2010}
\bysame, {\em An analog of {K}ostant's theorem for the cohomology of quantum
  groups}, Proc. Amer. Math. Soc. \textbf{138} (2010), no.~1, 85--99.

\end{thebibliography}

\providecommand{\bysame}{\leavevmode\hbox to3em{\hrulefill}\thinspace}

\end{document}